\documentclass[a4paper,11pt,reqno]{customart}

\usepackage[utf8x]{inputenc}
\usepackage[margin=1in]{geometry}
\usepackage{verbatim}
\usepackage{color}
\usepackage[table]{xcolor}
\usepackage{listings}
\usepackage{amssymb,amsfonts,amsthm,amsmath}
\usepackage{url}
\usepackage{enumerate}
\usepackage{todonotes}
\usepackage[all]{xy}
\usepackage{multirow}
\usepackage{attachfile}
\usepackage{stmaryrd}
\usepackage{footnote}
\makesavenoteenv{tabular}
\usepackage{hyperref}

\makeatletter
\newcommand\footnoteref[1]{\protected@xdef\@thefnmark{\ref{#1}}\@footnotemark}
\makeatother

\hypersetup{
	pdfborder={0 0 0}
}

\newcommand{\imp}{\rightarrow}

\newcommand{\Pb}{\mathbb{P}}
\newcommand{\Qb}{\mathbb{Q}}

\newcommand{\Psf}{\mathsf{P}}

\newcommand{\Acal}{\mathcal{A}}

\newcommand{\Ccal}{\mathcal{C}}
\newcommand{\Dcal}{\mathcal{D}}

\newcommand{\Fcal}{\mathcal{F}}
\newcommand{\Gcal}{\mathcal{G}}
\newcommand{\Ucal}{\mathcal{U}}

\newcommand{\Ical}{\mathcal{I}}

\newcommand{\Mcal}{\mathcal{M}}
\newcommand{\Pcal}{\mathcal{P}}

\newcommand{\uh}{{\upharpoonright}}

\newcommand{\halts}{{\downarrow}}

\renewcommand{\setminus}{\smallsetminus}

\newcommand{\dbf}{\mathbf{d}}

\def\qt#1{``#1''}%

\newcommand{\s}[1]{\ensuremath{\sf{#1}}}

\DeclareMathOperator{\rca}{\s{RCA}_0}
\DeclareMathOperator{\piooca}{\Pi^1_1\s{CA}}
\DeclareMathOperator{\aca}{\s{ACA}}
\DeclareMathOperator{\atr}{\s{ATR}}
\DeclareMathOperator{\wkl}{\s{WKL}}

\DeclareMathOperator{\rt}{\s{RT}}

\DeclareMathOperator{\rrt}{\s{RRT}}

\DeclareMathOperator{\coh}{\s{COH}}

\DeclareMathOperator{\ts}{\s{TS}}

\DeclareMathOperator{\ipt}{\s{IPT}}

\DeclareMathOperator{\fs}{\s{FS}}

\newcommand{\qvdash}{\operatorname{{?}{\vdash}}}
\newcommand{\nqvdash}{\operatorname{{?}{\nvdash}}}

\newtheoremstyle{custom}
  {10pt}
  {10pt}
  {\normalfont}
  {}
  {\bfseries}
  {}
  { }
  {}

\theoremstyle{custom}

\usepackage{xcolor}	
\usepackage{soul}

\newtheorem{theorem}{Theorem}[section]
\newtheorem{lemma}[theorem]{Lemma}

\newtheorem{corollary}[theorem]{Corollary}

\theoremstyle{definition}
\newtheorem{definition}[theorem]{Definition}

\newtheorem*{statement}{Statement}

\theoremstyle{remark}

\newtheorem {question}[theorem]{Question}

\numberwithin{equation}{section}

\title{Pigeons do not jump high}
\author{
  Benoit Monin \and Ludovic Patey
}

\begin{document}

\begin{abstract}
The infinite pigeonhole principle for 2-partitions asserts the existence, for every set $A$, of an infinite subset of $A$ or of its complement. In this paper, we develop a new notion of forcing enabling a fine analysis of the computability-theoretic features of the pigeonhole principle. We deduce various consequences, such as the existence, for every set $A$, of an infinite subset of it or its complement of non-high degree. We also prove that every $\Delta^0_3$ set has an infinite low${}_3$ solution and give a simpler proof of Liu's theorem that every set has an infinite subset in it or its complement of non-PA degree.
\end{abstract}

\maketitle

\section{Introduction}

The infinite pigeonhole principle asserts the existence, for any $k$-partition of the integers, of an infinite subset of one of the parts. In particular, the pigeonhole principle for 2-partitions asserts that, for every set $A$, there is an infinite subset of $A$ or of $\overline{A}$. The pigeonhole principle can be seen as a mathematical problem, with instances and solutions. An \emph{instance} is a $k$-partition of the integers, and a \emph{solution} to an instance is an infinite subset of one of the parts. In this paper, we conduct a computability-theoretic study of the pigeonhole principle seen as a problem. More precisely, given an arbitrary instance of the pigeonhole principle, we show the existence of a \qt{weak} solution, for various computability-theoretic notions of weakness. Our main motivation is reverse mathematics.

\subsection{Reverse mathematics and Ramsey's theorem}

Reverse mathematics is a foundational program which seeks to determine the optimal axioms to prove ordinary theorems. It uses the framework of second-order arithmetics, with a base theory, $\rca$, capturing “computable mathematics". The early study of reverse mathematics revealed the existence of four linearly ordered big systems $\wkl$, $\aca$, $\atr$, and $\piooca$ (in increasing order), such that, given an ordinary theorem, it is very likely either to be provable in $\rca$, or provably equivalent to one of the four systems in $\rca$. These systems together with $\rca$ are known as the \qt{Big Five}. Among them, $\wkl$ stands for \qt{weak K\"onig's lemma}, and asserts that every infinite binary tree admits an infinite path, while $\aca$ is the comprehension axiom restricted to arithmetical formulas. $\wkl$ can be thought of as capturing compactness arguments, and $\aca$ is equivalent to the existence, for every set $X$, of the halting set relative to~$X$. See Simpson~\cite{Simpson2009Subsystems} for an introduction to reverse mathematics.

Among the theorems studied in reverse mathematics, Ramsey's theorem received a special attention from the community, since Ramsey's theorem for pairs historically was the first theorem known to escape the Big Five phenomenon. Given a set of integers $X$, $[X]^n$ denotes the set of all $n$-tuples over $X$. For a coloring $f : [\omega]^n \to k$, a set of integers $H$ is \emph{homogeneous} if $f$ is constant over $[H]^n$.

\begin{statement}[Ramsey's theorem]
$\rt^n_k$: \qt{Every $k$-coloring of $[\omega]^n$ admits an infinite homogeneous set}.	
\end{statement}

In particular, $\rt^1_k$ is the infinite pigeonhole principle for $k$-partitions. Ramsey's theorem and its consequences are notoriously hard to analyse from a computability-theoretic viewpoint. Jockusch~\cite{Jockusch1972Ramseys} proved that $\rt^n_k$ is equivalent to $\aca$ whenever $n \geq 3$, thereby showing that $\rt^n_k$ satisfies the Big Five phenomenon. The question whether $\rt^2_k$ implies $\aca$ was a longstanding open question, until Seetapun~\cite{Seetapun1995strength} proved that $\rt^2_k$ is strictly weaker than $\aca$. Later, Jockusch~\cite{Jockusch1972Ramseys,Jockusch197201} and Liu~\cite{Liu2012RT22} showed that $\rt^2_k$ is incomparable with $\wkl$, and therefore that $\rt^2_k$ is not even linearly ordered with the Big Five. See Hirschfeldt~\cite{Hirschfeldt2015Slicing} for an introduction to the reverse mathematics of Ramsey's theorem.

\subsection{Cohesiveness and the pigeonhole principle}

Cholak, Jockusch and Slaman~\cite{Cholak2001strength} made a breakthrough in the understanding of Ramsey's theorem for pairs, by decomposing $\rt^2_2$ into a cohesiveness principle, and the pigeonhole principle for $\Delta^0_2$ instances. An infinite set $C$ is \emph{cohesive} for a countable sequence of sets $R_0, R_1, \dots$ if $C \subseteq^{*} R_i$ or $C \subseteq^{*} \overline{R}_i$ for every $i \in \omega$, where $\subseteq^*$ means inclusion but for finitely many elements.

\begin{statement}[Cohesiveness]
$\coh$: \qt{Every countable sequence of sets has a cohesive set}.	
\end{statement}

\begin{statement}
$\mathsf{D}^n_k$: \qt{For every $\Delta^0_n$ $k$-partition of $\omega$, there is an infinite subset of one of the parts}.
\end{statement}

Cholak, Jockusch and Slaman~\cite{Cholak2001strength}, Mileti~\cite{Mileti2004Partition}
and Chong, Lempp and Yang~\cite{Chong2010role}, proved that $\rt^2_2$ is equivalent to $\coh \wedge \mathsf{D}^2_2$. The interest of such a decomposition comes from the combinatorial simplicity of $\coh$ and $\mathsf{D}^2_2$. Indeed, $\coh$ can be seen as a sequential version of $\rt^1_2$ with finite errors, while $\mathsf{D}^2_2$ is $\rt^1_2$ for $\Delta^0_2$ instances. One may naturally wonder whether such a decomposition is strict, that is, whether both $\coh$ and $\mathsf{D}^2_2$ are strictly weaker than $\rt^2_2$ over $\rca$. Hirschfeldt, Jockusch, Kjoss-Hanssen, Lempp and Slaman~\cite{Hirschfeldt2008strength} proved that $\coh$ is strictly weaker than $\rt^2_2$ over $\rca$. Much later, Chong, Slaman and Yang~\cite{Chong2014metamathematics} proved that $\mathsf{D}^2_2$ is strictly weaker than $\rt^2_2$ over $\rca$, answering a long-standing open problem. However, the latter proof strongly relies on non-standard models, in that it constructs a model of $\rca + \mathsf{D}^2_2$ containing only low sets, that is, sets $X$ such that $X' \leq_T \emptyset'$. However, Downey, Hirschfeldt, Lempp and Solomon~\cite{Downey200102} constructed a $\Delta^0_2$ set with no low infinite subset of it or its complement. This shows that there cannot be an $\omega$-model of $\rca + \mathsf{D}^2_2$ with only low sets, where an $\omega$-structure is a structure whose first-order part consists of the standard integers. The following question is arguably the most important question in reverse mathematics, not only by its self interest, but also by range of related questions, new techniques and intellectual emulation it generated in the computability-theoretic community.

\begin{question}
Is every $\omega$-model of $\mathsf{D}^2_2$ a model of $\rt^2_2$?
\end{question}

This question is equivalent to asking whether every $\omega$-model of $\mathsf{D}^2_2$ is a model of $\coh$. A particular way to prove such an implication would be, given a sequence of sets $R_0, R_1, \dots$, to construct a $\Delta^{0,\vec{R}}_2$ set $A$ such that every infinite subset of $A$ or $\overline{A}$ computes relative to $\vec{R}$ a cohesive set for $\vec{R}$.
Among the instances of $\coh$, the sequence of primitive recursive sets $\vec{R}$ is maximally difficult, in that for every computable sequence of sets $\vec{S}$, every cohesive set for $\vec{R}$ computes a cohesive set for $\vec{S}$. The sets cohesive for the sequence of primitive recursive sets are called \emph{p-cohesive}. Jockusch and Stephan~\cite{Jockusch1993cohesive} studied the p-cohesive degrees, and proved that these are the precisely the degrees whose Turing jump is PA over $\emptyset'$. The following question is therefore of particular interest.

\begin{question}\label{quest:delta2-non-pazp}
Is there a $\Delta^0_2$ set $A$ such that for every infinite set $H \subseteq A$ or $H \subseteq \overline{A}$, the jump of $H$ is PA over $\emptyset'$?
\end{question}

A degree $\dbf$ is \emph{high} if $\dbf' \geq \mathbf{0}''$. A particular way to answer positively the previous question would be by proving that there is a $\Delta^0_2$ set $A$ whose solutions are of high degrees. However, Cholak, Jockusch and Slaman~\cite{Cholak2001strength} proved that given a non-$\Delta^0_2$ set $C$, every $\Delta^0_2$ set admits an infinite subset $H$ of it or its complement such that $C$ is not $\Delta^{0,H}_2$. In particular, we can always obtain a solution $H$ of non-high degree.

By an empirical observation, many proofs of the existence of \qt{weak} solutions for $\Delta^0_2$ instances of the pigeonhole principle are actually proofs of such an existence for arbitrary (even non-$\Delta^0_2$) instances of the pigeonhole principle. For instance, Dzhafarov and Jockusch~\cite{Dzhafarov2009Ramseys} proved the existence, for every set $A$ and every non-computable set $C$, of a solution to $A$, that is, an infinite subset of $A$, which does not compute $C$. Liu~\cite{Liu2012RT22} proved the existence of a solution of non-PA degree, and more generally of solutions computing no enumeration of a closed set in the Cantor space~\cite{Liu2015Cone}. The second author~\cite{Patey2016reverse} proved the existence, for every set $A$ and every hyperimmune function $f$, of a solution $H$ to $A$ such that $f$ is $H$-hyperimmune, where a function $f$ is \emph{$H$-hyperimmune} if it is not dominated by any $H$-computable function. This observation could provide a partial answer to the difficulty of answering Question~\ref{quest:delta2-non-pazp}. Maybe there exists a (non-necessarily $\Delta^0_2$) set $A$ such that every solution has a jump of PA degree over $\emptyset'$, or even of high degree. Then, any answer to Question~\ref{quest:delta2-non-pazp} would necessarily rely on $\Delta^0_2$ approximations of the set $A$. This motivates our first main theorem: 

\begin{theorem}
Every set $A$ has an infinite subset $H \subseteq A$ or $H \subseteq \overline{A}$
of non-high degree.
\end{theorem}
This theorem can be taken as a further evidence towards the intuition that Question~\ref{quest:delta2-non-pazp} does not depend on the $\Delta^0_2$ definability of the set $A$. Note that by an observation of the second author~\cite{Patey2016Open}, a negative answer to Question~\ref{quest:delta2-non-pazp} for non-$\Delta^0_2$ sets would have consequences on other statements studied in reverse mathematics, notably the increasing polarized Ramsey theorem for pairs ($\ipt^2_2$) introduced by Dzhafarov and Hirst~\cite{Dzhafarov2009polarized}.

\subsection{The hierarchies in reverse mathematics}

The computability-theoretic study of the pigeonhole principle is also motivated by questions on the strictness of hierarchies in reverse mathematics. Many consequences of Ramsey's theorem form hierarchies of statements, parameterized by the size of the colored tuples. A first example is Ramsey's theorem itself. Indeed, $\rt^{n+1}_k$ implies $\rt^n_k$ for every $n, k \geq 1$. By the work of Jockusch~\cite{Jockusch1972Ramseys}, this hierarchy collapses starting from the triples, and by Seetapun~\cite{Seetapun1995strength}, Ramsey's theorem for pairs is strictly weaker than Ramsey's theorem for triples. We therefore have 
$$
\rt^1_k < \rt^2_k < \rt^3_k = \rt^4_k = \dots
$$

Friedman~\cite{FriedmanFom53free} introduced the free set and thin set theorems in reverse mathematics, while Csima and Mileti~\cite{Csima2009strength} introduced and studied the rainbow Ramsey theorem. A coloring $f : [\omega]^n \to \omega$ is \emph{$k$-bounded} if each color occurs at most $k$ times. An infinite set of integers $H$ is \emph{thin} for $f$ if $f$ omits at least one color over $[H]^n$. We say that $H$ is \emph{free} for $f$ if for every $x \in H$, $H \setminus \{x\}$ is thin for $f$. Last, $H$ is a \emph{rainbow} for $f$ if each color occurs at most once on $[H]^n$.

\begin{statement}[Free set theorem]
$\fs^n$: \qt{Every coloring of $[\omega]^n$ admits an infinite free set}.
\end{statement}

\begin{statement}[Thin set theorem]
$\ts^n$: \qt{Every coloring of $[\omega]^n$ admits an infinite thin set}.
\end{statement}

\begin{statement}[Rainbow Ramsey theorem]
$\rrt^n_k$: \qt{Every $k$-bounded coloring of $[\omega]^n$ admits an infinite rainbow}.
\end{statement}

The reverse mathematics of these statements were extensively studied in the literature~\cite{Cholak2001Free,Csima2009strength,Kang2014Combinatorial,PateyCombinatorial,Patey2015Somewhere,Patey2016weakness,RiceThin,WangSome,Wang2013Rainbow,Wang2014Cohesive,Wang2014Definability,Wang2014Some}. In particular, these theorems form hierarchies which are not known to be strictly increasing over~$\rca$. 

\begin{question}\label{quest:strictness-hiearchies}
Are the hierarchies of the free set, thin set, and rainbow Ramsey theorem strictly increasing?
\end{question}

Partial results were however obtained. All these statements admit lower bounds of the form “For every $n \geq 2$, there is a computable instance of $\Psf^n$ with no $\Sigma^0_n$ solution", where $\Psf^n$ denotes any of $\rt^n_k$ (Jocksuch~\cite{Jockusch1972Ramseys}), $\rrt^n_k$ (Csima and Mileti~\cite{Csima2009strength}), $\fs^n$, or $\ts^n$ (Cholak, Giusto, Hirst and Jockusch~\cite{Cholak2001Free}). From the upper bound viewpoint, all these statements follow from Ramsey's theorem. Therefore, by Cholak, Jockusch and Slaman~\cite{Cholak2001strength}, every computable instance of $\Psf^1$ admits a computable solution, and every computable instance of $\Psf^2$ admits a low${}_2$ solution. These results are sufficient to show that $\Psf^1 < \Psf^2 < \Psf^3$ in reverse mathematics. This upper bound becomes too coarse at triples, since $\rt^3_2$ is equivalent to $\aca$, while Wang~\cite{Wang2014Some} surprisingly proved that $\Psf^n$ is strictly weaker than $\aca$ for every $n$ and $\Psf^n$ among $\fs^n$, $\ts^n$, and $\rrt^n_k$. In particular, Wang~\cite{Wang2014Cohesive} proved that every computable instance of $\rrt^3_k$ admits a low${}_3$ solution. The following question is still open. A positive answer would also answer positively Question~\ref{quest:strictness-hiearchies}.

\begin{question}\label{quest:instances-hierarchies-lown}
Does every computable instance of $\fs^n$, $\ts^n$, and $\rrt^n_k$ admit a low${}_n$ solution?
\end{question}

The known techniques to prove upper bounds to $\fs^n$, $\ts^n$, and $\rrt^n_k$,
are done by forcing with an inductive argument. This is in particular the case to prove that $\fs^n$, $\ts^n$ and $\rrt^n_k$ does not imply $\aca$ (Wang~\cite{Wang2014Some}), $\wkl$ (Patey~\cite{PateyCombinatorial}), and preserve multiple hyperimmunities (Patey~\cite{Patey2017Iterative}) for every $n$. The techniques are all obtained by proving the result for arbitrary instances of the pigeonhole principle, and then generalizing to other hierarchies by an inductive argument. In this paper, we therefore prove the following theorem, which introduces the machinery that hopefully will serve to answer positively Question~\ref{quest:instances-hierarchies-lown}.

\begin{theorem}
Every $\Delta^0_3$ set $A$ has an infinite subset $H \subseteq A$ or $H \subseteq \overline{A}$
of low${}_3$ degree.
\end{theorem}

This gives a partial answer to a question of Wang~\cite[Questions 6.1 and 6.2]{Wang2014Cohesive} and the second author~\cite[Question 5.4]{Patey2016Open} for the case $n = 3$.

\subsection{Definitions and notation}

A \emph{binary string} is an ordered tuple of bits $a_0, \dots, a_{n-1} \in \{0, 1\}$.
The empty string is written $\epsilon$. A \emph{binary sequence} (or real) is an infinite listing of bits $a_0, a_1, \dots$.
Given $s \in \omega$,
$2^s$ is the set of binary strings of length $s$ and
$2^{<s}$ is the set of binary strings of length $<s$. As well,
$2^{<\omega}$ is the set of binary strings
and $2^{\omega}$ is the set of binary sequences.
Given a string $\sigma \in 2^{<\omega}$, we use $|\sigma|$ to denote its length.
Given two strings $\sigma, \tau \in 2^{<\omega}$, $\sigma$ is a \emph{prefix}
of $\tau$ (written $\sigma \preceq \tau$) if there exists a string $\rho \in 2^{<\omega}$
such that $\sigma \rho = \tau$. Given a sequence $X$, we write $\sigma \prec X$ if
$\sigma = X \uh n$ for some $n \in \omega$.
A binary string $\sigma$ can be interpreted as a finite set $F_\sigma = \{ x < |\sigma| : \sigma(x) = 1 \}$. We write $\sigma \subseteq \tau$ for $F_\sigma \subseteq F_\tau$.
We write $\#\sigma$ for the size of $F_\sigma$. 

A \emph{binary tree} is a set of binary strings $T \subseteq 2^{<\omega}$ which is closed downward under the prefix relation. A \emph{path} through $T$ is an binary sequence $P \in 2^\omega$ such that every initial segment belongs to $T$. 

A \emph{Turing ideal} $\Ical$ is a collection of sets which is closed downward under the Turing reduction and closed under the effective join, that is, $(\forall X \in \Ical)(\forall Y \leq_T X) Y \in \Ical$ and $(\forall X, Y \in \Ical) X \oplus T \in \Ical$, where $X \oplus Y = \{ 2n : n \in X \} \cup \{ 2n+1 : n \in Y \}$. A \emph{Scott set} is a Turing ideal $\Ical$ such that every infinite binary tree $T \in \Ical$ has a path in $\Ical$. In other words, a Scott set is the second-order part of an $\omega$-model of $\rca + \wkl$.
A Turing ideal $\Mcal$ is \emph{countable coded} by a set $X$
if $\Mcal = \{ X_n : n \in \omega \}$ with $X = \bigoplus_n X_n$.
A formula is $\Sigma^0_1(\Mcal)$ (resp.\ $\Pi^0_1(\Mcal)$) if it is $\Sigma^0_1(X)$ (resp.\ $\Pi^0_1(X)$) for some $X \in \Mcal$.

Given two sets $A$ and $B$, we denote by $A < B$ the formula
$(\forall x \in A)(\forall y \in B)[x < y]$.
We write $A \subseteq^{*} B$ to mean that $A - B$ is finite, that is, 
$(\exists n)(\forall a \in A)(a \not \in B \imp a < n)$.
A \emph{$k$-cover} of a set $X$ is a sequence of sets $Y_0, \dots, Y_{k-1}$ such that $X \subseteq Y_0 \cup \dots \cup Y_{k-1}$.

\section{Main concepts}

The main contribution of this paper is a new notion of forcing enabling a finer analysis of the computability-theoretic aspects of the infinite pigeonhole principle. All the theorems obtained in Section~\ref{sect:applications} are direct applications of this notion of forcing by taking a sufficiently generic filter, or by an effectivization of the construction of a filter. In order to give a better grasp on the notion of forcing, we focus in this section on some essential features of its design.

\subsection{Forcing question}\label{subsect:forcing-question}

In computability theory, forcing is often specified by a partial order $(\Pb, \leq)$
of conditions. Each condition $c \in \Pb$ is given an interpretation $[c] \subseteq 2^\omega$, such that $[d] \subseteq [c]$ whenever $d \leq c$. Informally, $c$ can be seen as a partial approximation of the object we construct, and $[c]$ denotes the set of all possible objects which satisfy this partial approximation. Then, every filter $\Fcal$ induces a collection of sets $[\Fcal] = \bigcap_{c \in \Fcal} [c]$. Any set $G \in [\Fcal]$ is called a \emph{generic set}. Whenever the filter $\Fcal$ is sufficiently generic, $[\Fcal]$ is often a singleton $\{G\}$, in which case the generic set is uniquely determined. 

Such notions of forcing induce a \emph{forcing relation} $c \Vdash \varphi(G)$ defined over conditions $c \in \Pb$ and arithmetical formulas with one formal set parameter $\varphi(G)$. In particular, $c \Vdash \varphi(G)$ for a $\Delta^0_0$ formula
if $\varphi(G)$ holds for every set $G \in [c]$. The relation is defined inductively for more complex formulas, so that it satisfies the following main lemma:

\begin{lemma}\label{lem:forcing-relation-spec}
For every sufficiently generic filter $\Fcal$, every set $G \in [\Fcal]$ and every arithmetical formula $\varphi(G)$, $\varphi(G)$ holds if and only if $c \Vdash \varphi(G)$ for some condition $c \in \Fcal$.
\end{lemma}

The computability-theoretic properties of the generic sets are strongly related to the existence, for every condition $c \in \Pb$, of a $\Sigma^0_n$-definable relation $c \qvdash \varphi(G)$ over $\Sigma^0_n$ formulas $\varphi(G)$ which satisfies the following properties:

\begin{lemma}\label{lem:forcing-question-spec}
Let $c \in \Pb$ be a condition, and $\varphi(G)$ be a $\Sigma^0_n$ formula.
\begin{itemize}
	\item[(a)] If $c \qvdash \varphi(G)$, then there is some $d \leq c$ such that $d \Vdash \varphi(G)$.
	\item[(b)] If $c \nqvdash \varphi(G)$, then there is some $d \leq c$ such that $d \Vdash \neg \varphi(G)$.
\end{itemize}
\end{lemma}

Any forcing relation $\Vdash$ induces a forcing question $\qvdash$ defined by $c \qvdash \varphi(G)$ if and only if $(\exists d \leq c) d \Vdash \varphi(G)$. In the case of Cohen forcing, that is, forcing over binary strings with the suffix relation, the default forcing question has the good definitional properties, that is, deciding a $\Sigma^0_n$ formula is $\Sigma^0_n$. However, for many other notions of forcing, this forcing question is definitionally too complex, and one has to define custom forcing relations and forcing questions, to have the desired complexity.

For instance, consider the notion of forcing $(\Pb, \leq)$ whose conditions are infinite computable binary trees, and such that $S \leq T$ if $S \subseteq T$. The interpretation of $T$ is the collection $[T]$ of its paths. We can define a forcing relation for $\Sigma^0_1$ and $\Pi^0_1$ formulas as follows.

\begin{definition}
Let $\psi(G, x)$ be a $\Delta^0_0$ formula, and $T \in \Pb$.
\begin{itemize}
	\item[(a)] $T \Vdash (\exists x)\psi(G,x)$ if there is some $\ell \in \omega$ such that
	for every $\sigma \in T$ with $|\sigma| = \ell$, $\psi(\sigma, w)$ holds for some $w < \ell$.
	\item[(b)] $T \Vdash (\forall x)\psi(G,x)$ if for every $\sigma \in T$ and every $w < |\sigma|$, $\psi(\sigma, w)$ holds.
\end{itemize}
\end{definition}

First, note that if $T \Vdash \varphi(G)$ where $\varphi(G)$ is $\Sigma^0_1$ or $\Pi^0_1$, then $\varphi(G)$ will hold for every filter $\Fcal$ containing $T$, and every generic set $G$ for this filter. Then, define $T \qvdash (\exists x)\psi(G, x)$ to hold if and only if $T \Vdash (\exists x)\psi(G, x)$. Let's assume that the formula $\psi(G, x)$ is continuous, that is, if $\psi(\sigma, w)$ holds and $\sigma \prec \tau$, then $\psi(\tau, w)$ holds. If $T \qvdash (\exists x)\psi(G, x)$, then $T \Vdash (\exists x)\psi(G, x)$ by definition. If $T \nqvdash (\exists x)\psi(G, x)$, then the set $S = \{ \sigma \in T : (\forall w < |\sigma|) \neg \psi(\sigma, w) \}$ is an infinite subtree of $T$ such that $S \Vdash (\forall x)\neg \psi(G, x)$. Note that $T \qvdash (\exists x)\psi(G, x)$ is a $\Sigma^0_1$ formula, which satisfies Lemma~\ref{lem:forcing-question-spec}.

Having a forcing question whose definition has the same complexity as the formula it decides, yields a few preservation properties for free. Let $(\Pb, \leq)$ be a notion of forcing such that the relation $c \qvdash \varphi(G)$ is uniformly $\Sigma^0_n$ whenever $\varphi(G)$ is $\Sigma^0_n$, and satisfies Lemma~\ref{lem:forcing-relation-spec} and Lemma~\ref{lem:forcing-question-spec}. The following lemma holds.

\begin{lemma}
For every non-$\Sigma^0_n$ set $C$, and every $\Sigma^0_n$ formula $\varphi(G, x)$, the following set is dense in $(\Pb, \leq)$.
$$
D = \{ c \in \Pb : (\exists w \not \in C) c \Vdash \varphi(G, w) \vee (\exists w \in C) c \Vdash \neg \varphi(G, w) \}
$$
\end{lemma}
\begin{proof}
Fix a condition $c \in \Pb$.
Let $W = \{ w \in \omega : c \qvdash \varphi(G, w) \}$. By assumption, the set $W$ is $\Sigma^0_n$, while $C$ is not. Let $w \in W \Delta C = (W \setminus C) \cup (C \setminus W)$.
If $w \in W \setminus C$, then $c \qvdash \varphi(G, w)$, so by Lemma~\ref{lem:forcing-question-spec}(a), there is some $d \leq c$ such that $d \Vdash \varphi(G, w)$. If $w \in C \setminus W$,
then $c \nqvdash \varphi(G, w)$, so by Lemma~\ref{lem:forcing-question-spec}(b), there is some $d \leq c$ such that $d \Vdash \neg \varphi(G, w)$. In both cases, $d$ belongs to~$D$.
\end{proof}

Then, for every sufficiently generic set $G$, $C$ will not be $\Sigma^{0,G}_n$. This is the notion of preservation of non-$\Sigma^0_n$ definitions, introduced by Wang~\cite{Wang2014Definability}. In particular, if some set $C$ is not $\Delta^0_n$, then either $C$ or $\overline{C}$ is not $\Sigma^0_n$, so by the same reasoning, $C$ will not be $\Delta^{0,G}_n$ for every sufficiently generic set $G$.

In many cases, the forcing question is compact in the following sense:

\begin{definition}
A forcing question $\qvdash$ is \emph{compact} if for every $c \in \Pb$ and every formula $\psi(G, x)$, $c \qvdash (\exists x)\psi(G, x)$ if and only if 
there is a finite set $U$ such that $c \qvdash (\exists x \in U)\psi(G, x)$.
\end{definition}

In particular, the forcing question for Cohen forcing and for the notion of forcing with computable binary trees is compact (see Wang~\cite[Section 3.2]{Wang2014Definability} for a definition of the forcing with computable binary trees). This yields other preservation properties for free.
A function $g$ \emph{dominates} a function $f$ if $g(x) \geq f(x)$ for every $x \in \omega$.
Given a set $X$, a function $f$ is \emph{$X$-hyperimmune} if is not dominated by any $X$-computable function. Let $(\Pb, \leq)$ be a notion of forcing with a compact forcing question satisfying the previous properties. The following lemma holds.

\begin{lemma}
For every $n$, every $\emptyset^{(n)}$-hyperimmune function $f$ and every Turing functional $\Phi_e$, the following set is dense in $(\Pb, \leq)$.
$$
D = \{ c \in \Pb : (\exists w) c \Vdash \Phi_e^{G^{(n)}}(w) \uparrow \vee (\exists w) c \Vdash \Phi_e^{G^{(n)}}(w) < f(w) \}
$$
\end{lemma}
\begin{proof}
Fix a condition $c \in \Pb$.
Let $g$ be the partial $\emptyset^{(n)}$-computable function which on input $w$, searches for a finite set $U$ such that $c \qvdash (\exists x \in U) \Phi_e^{G^{(n)}}(w) \halts = x$. It it finds such a set, then $g(w) = \max U$. Otherwise, $g(w) \uparrow$. We have two cases. In the first case, $g$ is total. Then, by $\emptyset^{(n)}$-hyperimmunity of $f$, there is some $w$ such that $g(w) < f(w)$. Let $U$ be the finite set witnessing that $g(w)\halts$. In other words, $c \qvdash (\exists x \in U) \Phi_e^{G^{(n)}}(w) \halts = x$. By Lemma~\ref{lem:forcing-question-spec}(a), there is some $d \leq c$ such that $d \Vdash (\exists x \in U) \Phi_e^{G^{(n)}}(w) \halts = x$, hence $d \Vdash \Phi_e^{G^{(n)}}(w) \halts < f(w)$. If $g$ is partial, say $g(w)\uparrow$ for some $w$. Then by compactness of the forcing question, $c \nqvdash (\exists x)\Phi_e^{G^{(n)}}(w) \halts = x$. By Lemma~\ref{lem:forcing-question-spec}(a), there is some $d \leq c$ such that $d \Vdash \Phi_e^{G^{(n)}}(w)\uparrow$.
\end{proof}

Then, for every sufficiently generic set $G$, $f$ will be $G^{(n)}$-hyperimmune. This is the notion of preservation of hyperimmunity, introduced by the second author~\cite{Patey2017Iterative}.

Whenever the extension $d$ of Lemma~\ref{lem:forcing-question-spec} is obtained $\emptyset^{(n)}$-uniformly in $c$, one can effectivize the construction to obtain a $\emptyset^{(n)}$-computable filter while controlling the $n$th jump of the generic set $G$, and therefore obtain a set of low${}_n$ degree. In the case of Cohen forcing, this yields the existence of a  low 1-generic set, and in the case of the forcing with computable binary trees, this yields the low basis theorem (Jockusch~\cite{Jockusch197201}).

\subsection{Mathias forcing and the pigeonhole principle}

In this paper, given a set $A$, we want to build a “weak" infinite subset $H$ of $A$ or of $\overline{A}$. We actually construct two sets $G^0 \subseteq A$ and $G^1 \subseteq \overline{A}$ by a variant of Mathias forcing, and ensure that at least one of them is infinite and satisfies the desired weakness property. In order to obtain a forcing question with the good definitional complexity, we shall use a different notion of forcing depending on the complexity of the formulas we want to control. 

In the case of $\Sigma^0_1$ and $\Pi^0_1$ formulas, we fix a countable Scott set $\Mcal$, and use a notion of forcing whose conditions are tuples $(F^0, F^1, X)$, where $F^0 \subseteq A$ and $F^1 \subseteq \overline{A}$ are finite sets, and $X \in \Mcal$ is an infinite set such that $\max(F^0, F^1) < \min X$. Let $A^0 = A$ and $A^1 = \overline{A}$. Our setting is slightly different from Section~\ref{subsect:forcing-question} since each filter $\Fcal$ induces two generic sets $G^0$ and $G^1$, defined by $G^i = \bigcup \{ F^i : (F^0, F^1, X) \in \Fcal \}$ for each $i < 2$. A condition $c = (F^0, F^1, X)$ has therefore two interpretations $[c]^0$ and $[c]^1$, defined by $[c]^i = \{ H : F^i \subseteq H \subseteq (F^i \cup X) \cap A^i  \}$ for each $i < 2$. We also need to define two forcing relations depending on which of the generic sets $G^0$ and $G^1$ we control. The natural forcing relations are again too complex from a definitional point of view, and we need to define custom ones.

\begin{definition}
Let $\psi(G, x)$ be a $\Delta^0_0$ formula, $c = (F^0, F^1, X)$ and $i < 2$.
\begin{itemize}
	\item[(a)] $c \Vdash^i (\exists x)\psi(G,x)$ if there is some $w \in \omega$ such that
	$\psi(F^i, w)$ holds.
	\item[(b)] $c \Vdash^i (\forall x)\psi(G,x)$ if for every $w \in \omega$ and every $E \subseteq X$, $\psi(F^i \cup E, w)$ holds.
\end{itemize}
\end{definition}

Note that the definition of the forcing relation for $\Pi^0_1$ formulas is stronger than the canonical one, since it would suffice to require that $\psi(F^i \cup E, w)$ holds for every $w \in \omega$ and \emph{every $E \subseteq X \cap A^i$}. Because of this, it is not in general the case that,  given a $\Sigma^0_1$ formula $\varphi(G)$ and a side $i < 2$, the set of conditions $c$ such that $c \Vdash^i \varphi(G)$ or $c \Vdash^i \neg \varphi(G)$ is dense. However, Cholak, Jockusch and Slaman~\cite{Cholak2001strength} designed a disjunctive forcing question ensuring this property on at least one side. We now detail it.

\begin{definition}
Given a condition $c = (F^0, F^1, X)$ and two $\Sigma^0_1$ formulas $\varphi^0(G)$ and $\varphi^1(G)$, define $c \qvdash \varphi^0(G^0) \vee \varphi^1(G^1)$ to hold if for every 2-cover $Z^0 \cup Z^1 = X$, there is some side $i < 2$ and some finite set $E \subseteq Z^i$ such that $\varphi^i(F^i \cup E)$ holds.
\end{definition}

This forcing relation satisfies the following disjunctive property.

\begin{lemma}[Cholak, Jockusch and Slaman~\cite{Cholak2001strength}]
Let $c \in \Pb$ be a condition, and $\varphi^0(G)$ and $\varphi^1(G)$ be $\Sigma^0_1$ formulas.
\begin{itemize}
	\item[(a)] If $c \qvdash \varphi^0(G^0) \vee \varphi^1(G^1)$, then there is some $d \leq c$ and some $i < 2$ such that $d \Vdash^i \varphi^i(G)$.
	\item[(b)] If $c \nqvdash \varphi^0(G^0) \vee \varphi^1(G^1)$, then there is some $d \leq c$ and some $i < 2$ such that $d \Vdash^i \neg \varphi^i(G)$.
\end{itemize}
\end{lemma}
\begin{proof}
Suppose $c \qvdash \varphi^0(G^0) \vee \varphi^1(G^1)$ holds. Then letting $Z^0 = X \cap A^0$ and $Z^1 = X \cap A^1$, there is some side $i < 2$ and some finite set $E \subseteq X \cap A^i$ such that $\varphi^i(F^i \cup E)$ holds. The condition $d = (F^i \cup E, F^{1-i}, X \cap (\max E, \infty))$ is an extension of $c$ such that $d \Vdash^i \varphi^i(G)$.

Suppose now that $c \nqvdash \varphi^0(G^0) \vee \varphi^1(G^1)$. Let $\Pcal$ be the collection of all the 2-covers $Z^0 \cup Z^1 = X$ such that for every $i < 2$ and every finite set $E \subseteq Z^i$, $\varphi^i(F^i \cup E)$ does not hold. $\Pcal$ is a non-empty $\Pi^{0,X}_1$ class, so since $X \in \Mcal \models \wkl$, there is some 2-cover $Z^0 \cup Z^1 \in \Pcal \cap \Mcal$. Let $i < 2$ be such that $Z^i$ is infinite. Then the condition $d = (F^0, F^1, Z^i)$ is an extension of $c$ such that $d \Vdash^i \neg \varphi^i(G)$.
\end{proof}

By a pairing argument (if for every pair $m, n \in \omega$, $m \in A$ or $n \in B$, then $A = \omega$ or $B = \omega$), if a filter $\Fcal$ is sufficiently generic, there is some side $i$
such that for every $\Sigma^0_1$ formula $\varphi(G)$, there is some $c \in \Fcal$ such that 
$c \Vdash^i \varphi(G)$ or $c \Vdash^i \neg \varphi(G)$. We therefore get the following lemma.

\begin{lemma}
For every sufficiently generic filter $\Fcal$ and every set $G^i \in [\Fcal]^i$, there is a side $i < 2$ such that for every $\Sigma^0_1$ formula $\varphi(G)$, $\varphi(G^i)$ holds if and only if $c \Vdash^i \varphi(G)$ for some condition $c \in \Fcal$.
\end{lemma}

In this paper, we generalize the combinatorics of Cholak, Jockusch and Slaman~\cite{Cholak2001strength} to design a notion of forcing with a forcing question having the right definitional complexity for upper formulas. This generalization involves the development of some new forcing machineries.

\subsection{Largeness classes}

The combinatorics of Cholak, Jockusch and Slaman enable one to decide a $\Sigma^0_1$ formula relative to the generic set $G$ independently of the set $A$, by asking whether the $\Sigma^0_1$ formula holds over “sufficiently many" finite sets. We make this largeness criterion precise through the notion of largeness class.

\begin{definition}
A \emph{largeness class} is a collection of sets $\Acal \subseteq 2^\omega$ such that
\begin{itemize}
	\item[(a)] If $X \in \Acal$ and $Y \supseteq X$, then $Y \in \Acal$
	\item[(b)] For every $k$-cover $Y_0, \dots, Y_{k-1}$ of $\omega$, there is some $j < k$ such that $Y_j \in \Acal$.  
\end{itemize} 
\end{definition}

For example, the collection of all the infinite sets is a largeness class. Moreover, any superclass of a largeness class is again a largeness class. We shall exclusively consider largeness classes which are countable intersections of $\Sigma^0_1$ classes, and which contain only infinite sets. 
Fix an effective enumeration $\Ucal_0, \Ucal_1, \dots$ of all the $\Sigma^0_1$ classes
 upward-closed under the superset relation, that is, if $X \in \Ucal_e$ and $Y \supseteq X$, then $Y \in \Ucal_e$.
These largeness classes can be represented by sets of integers $C$, denoting the class $\bigcap_{e \in C} \Ucal_e$, where $\{\Ucal_e\}_{e \in \omega}$ is a standard enumeration of the $\Sigma^0_1$ classes. Let us illustrate how one uses largeness classes to force $\Pi^0_2$ facts.

\begin{definition} \label{def:the_not_riemann_zeta_function}
Let $\zeta : \omega \times 2^{<\omega} \times \omega \to \omega$ be the computable function that takes as a parameter a code for a $\Delta_0$ formula $\Phi_e(G, n, m)$, a string $\sigma$ and an integer $n$,
and which gives a code for the open set
$$
\{ X : (\exists \rho \subseteq X - \{ 0, \dots, |\sigma|\})(\exists m) \neg \Phi_e(\sigma \cup \rho, n, m) \}
$$
\end{definition}

Fix a $\Delta_0$ formula $\Phi_a(G, n, m)$.
Suppose that $C$ is a set of integers such that $\bigcap_{e \in C} \Ucal_e$ is a largeness class that contains only infinite sets, and such that for every finite sequence $\sigma$ and every $n$, $\zeta(a, \sigma, n) \in C$. 

Fix a set $A$, and let $A^0 = \overline{A}$ and $A^1 = A$. Since $\bigcap_{e \in C} \Ucal_e$ is a largeness class, there is some $i < 2$ such that $A^i \in \bigcap_{e \in C} \Ucal_e$.
We can then build an infinite subset $H$ of $A^i$ such that $(\forall n)(\exists m)\neg \Phi_a(H, n, m)$ holds by the finite extension method $\sigma_0 \subseteq \sigma_1 \subseteq \dots \subseteq A^i$, letting $H = \bigcup_s \sigma_s$. First, note that $A^i$ must be infinite since $\bigcap_{e \in C} \Ucal_e$ contains only infinite sets. Therefore, given an initial segment $\sigma_s \subseteq A^i$, one can find an extension $\sigma_{s+1} \succeq \sigma_s$ such that $\#\sigma_{s+1} > \#\sigma_s$. Then, given some $n \in \omega$ and an initial segment $\sigma_s \subseteq A^i$, since $\zeta(a, \sigma_s, n) \in C$, $A^i \in \Ucal_{\zeta(a, \sigma_s, n)}$, so there is some $\rho \subseteq A^i - \{ 0, \dots, |\sigma_s|\}$ and some $m \in \omega$ such that $\neg \Phi_a(\sigma_s \cup \rho, n, m)$ holds.
	Letting $\sigma_{s+1} = \sigma_s\rho$, we made some progress to satisfy the $\Pi^0_2$ fact $(\forall n)(\exists m)\neg \Phi_a(H, n, m)$.

Before moving to the design of the notions of forcing, we prove two technical lemmas about largeness classes.

\begin{lemma}\label{lem:decreasing-largeness-yields-largeness}
Suppose $\Acal_0 \supseteq \Acal_1 \supseteq \dots$ is a decreasing sequence of largeness classes.
Then $\bigcap_s \Acal_s$ is a largeness class.	
\end{lemma}
\begin{proof}
If $X \in \bigcap_s \Acal_s$ and $Y \supseteq X$, then for every $s$, since $\Acal_s$ is a largeness class, $Y \in \Acal_s$, so $Y \in \bigcap_s \Acal_s$.
Let $Y_0, \dots, Y_{k-1}$ be a $k$-cover of $\omega$. For every $s \in \omega$, there is some $j < k$
such that $Y_j \in \Acal_s$. By the infinite pigeonhole principle, there is some $j < k$ such that $Y_j \in \Acal_s$ for infinitely many $s$. Since $\Acal_0 \supseteq \Acal_1 \supseteq \dots$ is a decreasing sequence,  $Y_j \in \bigcap_s \Acal_s$.
\end{proof}

\begin{lemma}\label{lem:largeness-class-complexity}
Let $\Acal$ be a $\Sigma^0_1$ class. 
The sentence “$\Acal$ is a largeness class" is $\Pi^0_2$.
\end{lemma}
\begin{proof}
Say $\Acal = \{ X : (\exists \sigma \preceq X)\varphi(\sigma) \}$ where $\varphi$ is a $\Sigma^0_1$ formula.
By compactness, $\Acal$ is a largeness class iff for every $\sigma$ and $\tau$ such that $\sigma \subseteq \tau$ and $\varphi(\sigma)$ holds, $\varphi(\tau)$ holds, and for every $k$, there is some $n \in \omega$ such that for every $\sigma_0 \cup \dots \cup \sigma_{k-1} = \{0, \dots, n\}$, there is some $j < k$ such that $\varphi(\sigma_j)$ holds.
\end{proof}

\subsection{From Mathias forcing to a second jump control}

The notion of forcing used to control the first jump of solutions to the infinite pigeonhole principle is a variant of Mathias forcing, a purely combinatorial notion with no effectiveness restriction on the reservoirs. This notion is essential in the study of Ramsey's theory. We now review the basic definitions of Mathias forcing, and then describe how to enrich this notion of forcing to have a better second jump control.

\begin{definition}
Let $\Qb_0$ be the set of ordered pairs $(\sigma, X)$ such that $X$ is infinite and $X \cap \{ 0, \dots, |\sigma|\} = \emptyset$. 
\end{definition}

Mathias forcing builds a single object $G$ by approximations (conditions) which consist in an initial segment $\sigma$ of $G$, and an infinite reservoir of integers. The purpose of the reservoir is to restrict the set of elements we are allowed to add to the initial segment. The reservoir therefore enriches the standard Cohen forcing by adding an infinitary negative restrain. The denotation of a condition is therefore naturally defined as follows.

Given a condition $p = (\sigma, X) \in \Qb_0$, let 
$$
[\sigma, X] = \{ Y \in [\omega]^\omega : \sigma \preceq Y \wedge Y - \{0, \dots, |\sigma| \} \subseteq X \}
$$

\begin{definition}
The partial order on $\Qb_0$ is defined by $(\tau, Y) \leq (\sigma, X)$ 
if $\sigma \preceq \tau$, $Y \subseteq X$ and $\tau - \sigma \subseteq X$.
\end{definition}

The following lemma is standard, and expresses that whenever the approximation becomes more precise, then the set of ``candidates" decreases.

\begin{lemma}\label{lem:qb0-denotation-compatible-order}
Suppose $(\tau, Y) \leq (\sigma, X) \in \Qb_0$. Then $[\tau, Y] \subseteq [\sigma, X]$.
\end{lemma}

The forcing relation for $\Sigma^0_1$ and $\Pi^0_1$ formulas can be defined in a natural way, and  has the right definitional complexity (relative to the reservoir), that is, forcing a $\Sigma^0_1$ and a $\Pi^0_1$ fact is $\Sigma^0_1$ and $\Pi^{0}_1$ relative to the reservoir, respectively. The relation can be extended to arbitrary arithmetical formulas by an inductive definition, but then forcing a $\Pi^0_n$ formula becomes $\Pi^0_{n+1}$ relative to the reservoir. This makes the forcing question for higher formula fail to have a good definitional complexity, even when the reservoir is required to be computable. We refer the reader to Cholak, Dzhafarov, Hirst and Slaman~\cite{Cholak2014Generics} for the study a computable Mathias forcing.

\begin{definition}
Let $\Phi_e(G, n)$ be a $\Delta_0$ formula with free variable $m$.
Let $p = (\sigma, X) \in \Qb_0$. 
\begin{itemize}
	\item[(a)] $p \Vdash (\exists n)\Phi_e(G, n)$ if $(\exists n)\Phi_e(\sigma, n)$ 
	\item[(b)] $p \Vdash (\forall n)\Phi_e(G, n)$ if $(\forall \tau \subseteq X)(\forall n)\Phi_e(\sigma \cup \tau, n)$
\end{itemize}
\end{definition}

\begin{lemma}\label{lem:qb0-forcing-holds}
Let $\Phi_e(G, n)$ be a $\Delta_0$ formula with free variable $n$, and let $p \in \Qb_0$.
\begin{itemize}
	\item[(a)] If $p \Vdash (\exists n)\Phi_e(G, n)$, then $(\exists n)\Phi_e(Y, n)$ holds for every $Y \in [p]$
	\item[(b)] If $p \Vdash (\forall n)\neg \Phi_e(G, n)$, then $(\forall n)\neg \Phi_e(Y, n)$ holds for every $Y \in [p]$.
\end{itemize}	
\end{lemma}


The forcing relation for $\Sigma^0_2$ formulas $(\exists n)(\forall m)\Phi_e(G, n, m)$ can be defined with Mathias forcing as $p \Vdash (\exists n)(\forall m)\Phi_e(G, n, m)$ iff $(\exists n)p \Vdash (\forall m)\Phi_e(G, n, m)$, and has the right definitional properties. The issue comes when considering $\Pi^0_2$ formulas $(\forall n)(\exists m)\neg \Phi_e(G, n, m)$. Forcing a $\Pi^0_2$ fact can be seen as a promise to satisfy a countable collection of $\Sigma^0_1$ facts. Since forcing a $\Sigma^0_1$ fact usually requires to take an extension, we cannot force all the $\Sigma^0_1$ facts simultaneously. A $\Pi^0_2$ fact is then forced if, whatever the further stage of the construction, it will always be possible to make some progress by forcing one more $\Sigma^0_1$ fact. 

In the case of Mathias forcing, the notion of reservoir is too permissive, and it is not possible to talk about the extensions of a condition with a definitionally simple formula. We will therefore enrich the notion of Mathias condition to add some restrictions on the reservoir, so that the extensions can be described in a simpler way.

\begin{definition}
Let $\Qb_1$ be the set of tuples $(\sigma, X, C, U)$ such that 
\begin{itemize}
	\item[(a)] $X \cap \{ 0, \dots, |\sigma|\} = \emptyset$ ; $X \supseteq U$ with $X,U \subseteq \omega$
	\item[(b)] $\bigcap_{e \in C} \Ucal_e$ is a largeness class containing only infinite sets
	\item[(c)] $U \in \bigcap_{e \in C} \Ucal_e$.
\end{itemize}
\end{definition}

One can think of a condition $(\sigma, X, C, U)$ as a Mathias condition $(\sigma, U)$ with a set $C$ denoting a largeness class $\bigcap_{e \in C} \Ucal_e$ which will impose some constraints on the nature of the reservoirs. This view is reflected through the denotation of a condition. Given a condition $p = (\sigma, X, C, U) \in \Qb_1$, let 
$$
[\sigma, X, C, U] = [\sigma, U]
$$

From a purely combinatorial viewpoint, the reservoir $X$ of a condition $(\sigma, X, C, U)$ could have been dropped, yielding a notion of forcing with 3-tuples $(\sigma, C, U)$. The reservoir $X$ is kept for effectiveness restrictions reasons, which will become clear in section~\ref{sect:forcing-rt12}. Indeed, the reservoir $X$ is responsible for forcing $\Pi^0_1$ facts, while the reservoir $U$ will force (together with $C$) $\Pi^0_2$ facts. When considering effective forcing, we shall see that $U$ will be “one jump up" of $X$. For example, $X$ can be taken to be of low degree, while $U$ will be low over $\emptyset'$. Since $X$ will solely be responsible for forcing $\Pi^0_1$ fact, we shall relate a condition $(\sigma, X, C, U)$ with the Mathias condition $(\sigma, X)$.

\begin{lemma}\label{lem:qb1-to-qb0-compatibility}
Suppose $(\sigma, X, C, U) \in \Qb_1$. Then
\begin{itemize}
	\item[(a)] $(\sigma, X) \in \Qb_0$
	\item[(b)] $[\sigma, X, C, U] \subseteq [\sigma, X]$
\end{itemize}
\end{lemma}
\begin{proof}
(a) Since $U \in \bigcap_e \Ucal_e$ and $\bigcap_e \Ucal_e$ contains only infinite sets, then $U$ is infinite. In particular, $X$ is infinite since $X \supseteq U$.
Moreover, $X \cap \{0, \dots, |\sigma|\} = \emptyset$. Therefore $(\sigma, X) \in \Qb_0$.
(b) Since $(\sigma, U) \leq (\sigma, X)$ as a Mathias condition,
by Lemma~\ref{lem:qb0-denotation-compatible-order}, $[\sigma, X, C, U] \subseteq [\sigma, X]$.
\end{proof}

In particular, if $(\sigma, X) \Vdash (\exists n)\Phi_e(G, n)$ (resp.\ $(\sigma, X) \Vdash (\forall n)\Phi_e(G, n)$), then $(\exists n)\Phi_e(Y, n)$ (resp.\ $(\forall n)\Phi_e(Y, n)$) holds for every $Y \in [\sigma, X, C, U]$.

\begin{definition}
The partial order on $\Qb_1$ is defined by $(\tau, Y, D, V) \leq (\sigma, X, C, U)$ 
if $\sigma \preceq \tau$, $Y \subseteq X$, $V \subseteq U$, $C \subseteq D$ and $\tau - \sigma \subseteq U$.
\end{definition}

From the definition of a forcing condition $(\sigma, X, C, U)$, and especially from the constraint that $U \in \bigcap_{e \in C} \Ucal_e$, it is not clear at all that there exists infinite decreasing sequences of conditions with non-trivial reservoirs, that is, with $U$ being coinfinite. In general, being a valid condition is not even closed under removing finitely elements from the reservoirs. Indeed, if $U \in \bigcap_{e \in C} \Ucal_e$ and $Y \subseteq X$ is cofinite in $X$, then it might be that $Y \cap U \not \in \bigcap_{e \in C} \Ucal_e$. Thankfully, since $\bigcap_{e \in C} \Ucal_e$ is a largeness class, we shall see in section~\ref{sect:forcing-rt12} that by carefully choosing our reservoirs, we will be able to apply some basic operations on them and keep having valid conditions.

\begin{lemma}\label{lem:qb1-order-compatibility-qb0}
Suppose $(\tau, Y, D, V) \leq (\sigma, X, C, U) \in \Qb_1$. Then
\begin{itemize}
	\item[(a)] $[\tau, Y, D, V] \subseteq [\sigma, X, C, U]$
	\item[(b)] $(\tau, Y) \leq (\sigma, X)$
\end{itemize}
\end{lemma}
\begin{proof}
(a)
Since $\sigma \preceq \tau$, $V \subseteq U$ and $\tau - \sigma \subseteq U$, then $(\tau, V) \leq (\sigma, U)$. Therefore $[\tau, Y, D, V] \subseteq [\sigma, X, C, U]$.
(b) Immediate since $(\tau, V) \leq (\sigma, U)$, $Y \subseteq X$ and and $\tau-\sigma \subseteq U \subseteq X$.
\end{proof}

We now define the forcing relation for $\Sigma^0_2$ formulas and $\Pi^0_2$ formulas.
In the case of $\Sigma^0_2$ formulas, this coincides with the forcing relation for $\Sigma^0_2$ formulas over Mathias forcing. The the case of $\Pi^0_2$ formulas is new, and is justified by our explanations about the combinatorics of largeness classes. Recall the function $\zeta$ of Definition~\ref{def:the_not_riemann_zeta_function}.

\begin{definition}\label{def:qb-forcing-relation}
Let $\Phi_e(G, n, m)$ be a $\Delta_0$ formula with free variables $m$ and $n$.
Let $p = (\sigma, X, C, U) \in \Qb_1$. 
\begin{itemize}
	\item[(a)] $p \Vdash (\exists n)(\forall m)\Phi_e(G, n, m)$ if $(\exists n)(\forall \tau \subseteq X)(\forall m)\Phi_e(\sigma \cup \tau, n, m)$ 
	\item[(b)] $p \Vdash (\forall n)(\exists m)\neg \Phi_e(G, n, m)$ if $(\forall \rho \subseteq U)(\forall n)\zeta(e, \sigma \cup \rho, n) \in C$
\end{itemize}
\end{definition}

\begin{lemma}\label{lem:qb1-forcing-closed-under-extension}
Let $\Phi_e(G, n, m)$ be a $\Delta_0$ formula with free variables $m$ and $n$.
Let $p, q \in \Qb_1$ be such that $q \leq p$. 
\begin{itemize}
	\item[(a)] If $p \Vdash (\exists n)(\forall m)\Phi_e(G, n, m)$ then $q \Vdash (\exists n)(\forall m)\Phi_e(G, n, m)$
	\item[(b)] If $p \Vdash (\forall n)(\exists m)\neg \Phi_e(G, n, m)$ then $q \Vdash (\forall n)(\exists m)\neg \Phi_e(G, n, m)$
\end{itemize}
\end{lemma}
\begin{proof}
Say $p = (\sigma, X, C, U)$ and $q = (\tau, Y, D, V)$
\begin{itemize}
	\item[(a)] Since $p \Vdash (\exists n)(\forall m)\Phi_e(G, n, m)$, then there is some $n$ such that $(\sigma, X) \Vdash (\forall m)\Phi_e(G, n, m)$. By Lemma~\ref{lem:qb1-order-compatibility-qb0}, $(\tau, Y) \leq (\sigma, X)$, so $(\tau, Y) \Vdash (\forall m)\Phi_e(G, n, m)$, hence $q \Vdash (\exists n)(\forall m)\allowbreak\Phi_e(G, n, m)$.
	\item[(b)] Let $\rho = \tau - \sigma$. By definition of $q \leq p$, $\rho \subseteq U$. Let $\rho_1 \subseteq V$. In particular, $\rho \cup \rho_1 \subseteq U$. By definition of $p \Vdash (\forall n)(\exists m)\neg \Phi_e(G, n, m)$, for every $n$, $\zeta(e, \sigma \cup \rho \cup \rho_1, n) \in C \subseteq D$. So $\zeta(e, \tau \cup \rho_1, n) \in D$ for every $n$. 
\end{itemize}
\end{proof}

We now define the notion of genericity which will be sufficient
to prove the main property of the forcing relation, that is,
whenever a formula is forced, then it will hold over the generic set.

\begin{definition}
A $\Qb_1$-filter $\Fcal$ is \emph{2-generic} if for every $\Sigma^0_2$ formula $\varphi(G)$,
there is some $p \in \Fcal$ such that $p \Vdash \varphi(G)$ or $p \Vdash \neg \varphi(G)$.
\end{definition}

As explained above, it is not clear at all that 2-generic $\Qb_1$-filters exist.
Their existence will be proven in section~\ref{sect:forcing-rt12}. 

\begin{lemma}\label{lem:qb1-forcing-pi2-progress}
Let $\Fcal$ be a 2-generic $\Qb_1$-filter, and $\Phi_e(G, n, m)$ be a $\Delta_0$ formula with free variables $m$ and $n$.
If $p \Vdash (\forall n)(\exists m)\neg \Phi_e(G, n, m)$ for some $p \in \Fcal$,
then for every $n \in \omega$, there is some $q = (\tau, Y, D, V) \in \Fcal$
such that $(\tau, Y) \Vdash (\exists m)\neg \Phi_e(G, n, m)$.
\end{lemma}
\begin{proof}
Fix $n \in \omega$, and let $\Phi_u(G, a, b) = \neg \Phi_e(G, n, a)$.
Since $\Fcal$ is a 2-generic filter, there is some $q = (\tau, Y, D, V) \in \Fcal$ such that
$$
q \Vdash (\exists a)(\forall b)\Phi_u(G, a, b) \mbox{ or } 
q \Vdash (\forall a)(\exists b)\neg \Phi_u(G, a, b)
$$
Suppose first $q \Vdash (\exists a)(\forall b)\Phi_u(G, a, b)$. Then in particular $(\tau, Y) \Vdash (\exists a)\neg \Phi_e(G, n, a)$ and we are done.
Suppose now $q \Vdash (\forall a)(\exists b)\neg \Phi_u(G, a, b)$. Since $\Fcal$ is a filter, we can assume that $q \leq p$. In particular, by the lemma's hypothesis, $q \Vdash (\forall n)(\exists m)\neg \Phi_e(G, n, m)$, so $\zeta(e, \tau, n) \in D$. Since $V \in \bigcap_{e \in D} \Ucal_e$, then $V \in \Ucal_{\zeta(e, \tau, n)}$. Therefore, there is some $\rho \subseteq V - \{0, \dots, |\tau|\}$ and some $m \in \omega$ such that $\neg \Phi_e(\tau \cup \rho, n, m)$ holds. Since $q \Vdash (\forall a)(\exists b)\neg \Phi_u(G, a, b)$, $\zeta(u, \tau \cup \rho, m) \in D$, so $V \in \Ucal_{\zeta(u, \tau \cup \rho, m)}$. Therefore, there is some $\mu \subseteq V - \{0, \dots, |\tau \cup \rho|\}$ such that $\Phi_e(\tau \cup \rho \cup \mu, n, m)$ holds. This contradicts $\neg \Phi_e(\tau \cup \rho, n, m)$.
\end{proof}

\begin{lemma}\label{lem:qb1-forcing-infinity}
Let $\Fcal$ be a 2-generic $\Qb_1$-filter.
Then there is some $p \in \Fcal$ such that
$p \Vdash (\forall n)(\exists m)[m > n \wedge m \in G]$.
\end{lemma}
\begin{proof}
Let $\Phi_e(G, n, m) \equiv [m \leq n \vee m \not \in G]$.
Since $\Fcal$ is a 2-generic filter, there is some $p = (\sigma, X, C, U) \in \Fcal$ such that
$$
p \Vdash (\exists n)(\forall m)\Phi_e(G, n, m) \mbox{ or } 
p \Vdash (\forall n)(\exists m)\neg \Phi_e(G, n, m)
$$
Suppose for the sake of contradiction that the first case holds.
Then $(\exists n)(\forall \rho \subseteq X)(\forall m)[m \leq n \vee m \not \in \sigma \cup \rho]$.
This is impossible since $X$ is infinite. 
\end{proof}

Given a collection $\Fcal \subseteq \Qb_1$, we let $G_\Fcal = \bigcup \{ \sigma : (\sigma, X, C, U) \in \Fcal \}$. 

\begin{lemma}\label{lem:qb1-forcing-holds}
Let $\Fcal$ be a 2-generic $\Qb_1$-filter.
Let $\Phi_e(G, n, m)$ be a $\Delta_0$ formula with free variables $m$ and $n$,
and let $p \in \Fcal$.
\begin{itemize}
	\item[(a)] If $p \Vdash (\exists n)(\forall m)\Phi_e(G, n, m)$, then $(\exists n)(\forall m)\Phi_e(G_\Fcal, n, m)$ holds.
	\item[(b)] If $p \Vdash (\forall n)(\exists m)\neg \Phi_e(G, n, m)$, then $(\forall n)(\exists m)\neg \Phi_e(G_\Fcal, n, m)$ holds.
\end{itemize}
\end{lemma}
\begin{proof}
Say $p = (\sigma, X, C, U)$.
We first prove (b). Fix some $n$. By Lemma~\ref{lem:qb1-forcing-pi2-progress}, there is some $q = (\tau, Y, D, V) \in \Fcal$ such that $(\tau, Y) \Vdash (\exists m)\neg \Phi_e(G, n, m)$. By Lemma~\ref{lem:qb0-forcing-holds}(a) $(\exists m)\neg \Phi_e(G_\Fcal, n, m)$ holds.
We now prove (a). For this, we claim that $G_\Fcal \in [p]$. Indeed, by (b) of this lemma (which we already proved) and Lemma~\ref{lem:qb1-forcing-infinity}, $G_\Fcal$ is infinite. Moreover, $\sigma \preceq G_\Fcal$, and for every $\rho \subseteq G_\Fcal - \sigma$, there is some $q = (\tau, Y, D, V) \in \Fcal$ with $q \leq p$ such that $\rho \subseteq \tau$. In particular, $\rho \subseteq U$, so $G_\Fcal \in [p]$. By Lemma~\ref{lem:qb1-to-qb0-compatibility}(b), $G_\Fcal \in [\sigma, X]$. Since $p \Vdash (\exists n)(\forall m)\allowbreak\Phi_e(G, n, m)$, there is some $n \in \omega$ such that $(\sigma, X) \Vdash (\forall m)\Phi_e(G, n, m)$. By Lemma~\ref{lem:qb0-forcing-holds}(b), $(\forall m)\Phi_e(G_\Fcal, n, m)$ holds.   
\end{proof}

\begin{lemma}\label{lem:qb1-generic-yields-infinite}
Let $\Fcal$ be a 2-generic $\Qb_1$-filter. Then $G_\Fcal$ is infinite.
\end{lemma}
\begin{proof}
Immediate by Lemma~\ref{lem:qb1-forcing-infinity} and Lemma~\ref{lem:qb1-forcing-holds}.
\end{proof}

\section{Pigeonhole forcing}\label{sect:forcing-rt12}

We now design the actual notion of forcing used to construct solutions to the infinite pigeonhole principle. It can be seen as a tree version of the $\Qb_1$-forcing, with some effectiveness restrictions on the conditions.
Let $\Mcal \models \wkl$ be a countable Turing ideal, and let $A^0 \cup A^1 = \omega$.

\begin{definition}\label{def:rt12-forcing-conditions}
Let $\Pb_1$ denote the set of conditions $(\sigma^0_s, \sigma^1_s, X_s, C, U_s : s < k)$ such that
\begin{itemize}
	\item[(a)] $\sigma^i_s \subseteq A^i$ for every $s < k$
	\item[(b)] $X_s \cap \{ 0, \dots, \max_i |\sigma^i_s|\} = \emptyset$ ; $X_s \supseteq U_s$ for every $s < k$
	\item[(c)] $U_0, \dots, U_{k-1}$ is a $k$-cover of $\omega - \{ 0, \dots, \max_{i,s} |\sigma^i_s| \}$
	\item[(d)] $\bigcap_{e \in C} \Ucal_e$ is a largeness class containing only infinite sets
	\item[(e)] $\vec{X}', \vec{U}, C \in \Mcal$
\end{itemize}
\end{definition}

A Turing ideal $\Mcal = \{X_0, X_1, \dots \}$ is \emph{countable coded} by a set $B$ if $B = \bigoplus_i X_i$. An \emph{index} of some $Z \in \Mcal$ is then some $i \in \omega$ such that $Z = X_i$. Thanks to the notion of index, any $\Pb_1$-condition can be finitely presented as follows. An \emph{index} of a $\Pb_1$-condition $c = (\sigma^0_s, \sigma^1_s, X_s, C, U_s : s < k)$ is a tuple $(\sigma^0_s, \sigma^1_s, a_s, b, e_s : s < k)$ where $a_s$ is an index for $X_s$, $b$ an index for $C$ and $e_s$ is an index for $U_s$. Note that the existence of these indices is ensured by property (e) of Definition~\ref{def:rt12-forcing-conditions}.

\begin{definition}
The partial order on $\Pb_1$ is defined by 
$$
(\tau^0_s, \tau^1_s, Y_s, D, V_s : s < \ell) \leq (\sigma^0_s, \sigma^1_s, X_s, C, U_s : s < k)
$$
if there is a function $f : \ell \to k$ such that for every $i < 2$ and $s < \ell$, $\sigma^i_{f(s)} \preceq \tau^i_s$, $Y_s \subseteq X_{f(s)}$,
	$V_s \subseteq U_{f(s)}$, $C \subseteq D$ and $\tau^i_s - \sigma^i_{f(s)} \subseteq U_{f(s)}$.
\end{definition}

We can think of the $\Pb_1$-forcing as a tree version of the $\Qb_1$-forcing.
Given a $\Pb_1$-condition $c = (\sigma^0_s, \sigma^1_s, X_s, C, U_s : s < k)$,
each $s < k$ will be referred to as a \emph{branch} of $c$. 
Each branch $s$ represents two candidate $\Qb_1$-conditions $c^{[0,s]} = (\sigma^0_s, X_s, C, U_s)$ and $c^{[1,s]} = (\sigma^1_s, X_s, C, U_s)$. Actually, they will not be true $\Qb_1$-conditions in general, since there is no reason why $U_s$ would belong to $\bigcap_{e \in C} \Ucal_e$. For example, $U_s$ might be finite. We shall however see in Lemma~\ref{lem:rt12-has-valid-branch} that there must be at least one branch $s$ such that $c^{[0,s]}, c^{[1,s]} \in \Qb_1$.

The notion of condition extension enables to fork branches, according to the function $f$ witnessing the extension. We write $d \leq_f c$ if $d \leq c$ is witnessed by the function $f$. We say that the branch $t$ of $d$ \emph{refines} the branch $s$ of $c$
if $f(t) = s$.
We are interested in two particular kinds of extensions: the ones which do not fork any branch, and the ones which fork exactly one branch.
Given a condition $c = (\sigma^0_s, \sigma^1_s, X_s, C, U_s : s < k) \in \Pb_1$,
a \emph{simple extension} of $c$ is a condition $d \leq_f c$
such that $f$ is the identity function.
An \emph{$s$-extension} of $c$ is a condition $d \leq_f c$
such that $f^{-1}(t)$ is a singleton for every $t \neq s$.

\subsection{Validity and projections}

We now develop the framework which relates $\Pb_1$-forcing to $\Qb_1$-forcing.
Informally, a branch $s$ of a condition $c = (\sigma^0_s, \sigma^1_s, X_s, C, U_s : s < k) \in \Pb_1$ is a good candidate if either $c^{[0,s]}, c^{[1,s]} \in \Qb_1$ is a valid $\Qb_1$-condition.

\begin{definition}
Let $c = (\sigma^0_s, \sigma^1_s, X_s, C, U_s : s < k)$ be a $\Pb_1$-condition.
A branch $s$ is \emph{valid} in $c$ if $U_s \in \bigcap_{e \in C} \Ucal_e$.
\end{definition}

From the discussion above, it should be clear that if a branch $s$ is valid in $c$,
then $c^{[0,s]}, c^{[1,s]} \in \Qb_1$. We first prove that valid branches exist.

\begin{lemma}\label{lem:rt12-has-valid-branch}
Every $\Pb_1$-condition has a valid branch.
\end{lemma}
\begin{proof}
Let $c = (\sigma^0_s, \sigma^1_s, X_s, C, U_s : s < k) \in \Pb_1$.
Suppose for the sake of contradiction that for every $s < k$, $U_s \not \in \bigcap_{e \in C} \Ucal_e$.
Let $u \in \omega$ be large enough to that  $\{U_s : s < k\} \cup \{ \{0, \dots,u \}\}$
is a cover of $\omega$. This cover contradicts the fact that $\bigcap_{e \in C} \Ucal_e$ is a largeness class containing only infinite sets.
\end{proof}

Thanks to compactness, being a valid branch is not definitionally too complex.
In particular, if we work within a Turing ideal countable coded by a set $B$ such that $B' \leq \emptyset''$, then being a valid branch is $\Pi^0_3$.

\begin{lemma}\label{lem:rt12-validity-complexity}
Let $c$ be a $\Pb_1$-condition.
The sentence “The branch $s$ is valid in $c$" is $\Pi^0_2(\Mcal)$.
\end{lemma}
\begin{proof}
Say $c = (\sigma^0_s, \sigma^1_s, X_s, C, U_s : s < k)$.
Then $s$ is valid in $c$ if for every $n \in C$, there is a finite set $E \subseteq U_s$ such that $E \in \bigcap_{e \in C, e < n} \Ucal_e$. The sentence is $\Pi^0_2(C \oplus X_s \oplus U_s)$, 
hence $\Pi^0_2(\Mcal)$.
\end{proof}


By upward-closure of the notion of largeness class, if a branch $t$ of a $\Pb_1$-condition $d$ is valid, and $d \leq_f c$, then the branch $f(t)$ of $c$ is also valid. Therefore, given an infinite decreasing sequence of conditions, the valid branches form an infinite subtree.

\begin{lemma}\label{lem:pb1-to-qb1-compatibility}
Suppose $d \leq_f c \in \Pb_1$ and $d^{[i,s]} \in \Qb_1$.
Then 
\begin{itemize}
	\item[(a)] $c^{[i, f(s)]} \in \Qb_1$
	\item[(b)] $d^{[i,s]} \leq c^{[i, f(s)]}$.
\end{itemize}
\end{lemma}
\begin{proof}
Say $c = (\sigma^0_s, \sigma^1_s, X_s, C, U_s : s < k)$ and $d = (\tau^0_s, \tau^1_s, Y_s, D, V_s : s < \ell)$.

(a) $X_s \cap \{ 0, \dots, |\sigma^i_s|\} = \emptyset$ and $\bigcap_{e \in C} \Ucal_e$ is a largeness class containing only infinite sets.
We need to check that $U_{f(s)} \in \bigcap_{e \in C} \Ucal_e$.
By assumption, $V_s \in \bigcap_{e \in D} \Ucal_e$.
Since $V_s \subseteq U_{f(s)}$, then $U_{f(s)} \in \bigcap_{e \in D} \Ucal_e$. Moreover, $C \subseteq D$, so $\bigcap_{e \in D} \Ucal_e \subseteq \bigcap_{e \in C} \Ucal_e$, and we are done.

(b) This is immediate by definition of the extension relation of $\Pb_1$.
\end{proof}


We now define the notion of projector, which in the context of an infinite decreasing sequence of $\Pb_1$-conditions, corresponds to an infinite path through the tree of valid branches.

\begin{definition}
Let $\Fcal \subseteq \Pb_1$ be a collection. An \emph{$\Fcal$-projector}
is a function $P : \Fcal \to \omega$ such that
\begin{itemize}
	\item[(a)] The branch $P(c)$ is valid in $c$ for every $c \in \Fcal$
	\item[(b)] If $d \leq_f c \in \Fcal$ then $P(c) = f(P(d))$.
\end{itemize}
According to our notation, we write $\Fcal^{[i,P]} = \{ c^{[i, P(c)]} : c \in \Fcal \}$. An \emph{$\Fcal$-projection} is a collection $\Fcal^{[i,P]} = \{ c^{[i, P(c)]} : c \in \Fcal \}$
for some $\Fcal$-projector $P$.
\end{definition}

\subsection{The forcing question}\label{subsect:rt12-forcing-question}

We now design a disjunctive forcing question to control the second jump, in the same spirit as the one designed by Cholak, Jockusch and Slaman~\cite{Cholak2001strength} to control the first jump. Given a branch $s$ of a condition $c$ and two $\Sigma^0_2$ formulas $\varphi_0(G)$ and $\varphi_1(G)$, we define a $\Sigma^0_1(\Mcal)$ relation $c \qvdash_s \varphi_0(G) \vee \varphi_1(G)$, such that
\begin{itemize}
	\item If $c \qvdash_s \varphi_0(G) \vee \varphi_1(G)$, then there is an extension $d$ such that for every projector $P$ going threw the branch $s$ of $c$ ($P(c) = s$), then $d^{[i,P(d)]} \Vdash \varphi_i(G)$ for some $i < 2$.
	\item If  $c \nqvdash_s \varphi_0(G) \vee \varphi_1(G)$, then there is an extension $d$ such that for every projector $P$ going threw the branch $s$ of $c$ ($P(c) = s$), then $d^{[i,P(d)]} \Vdash \neg \varphi_i(G)$ for some $i < 2$.
\end{itemize}
Fix a sufficiently generic $\Pb_1$-filter $\Fcal$ and an $\Fcal$-projector $P$.
By a pairing argument and using the disjunctive forcing question, there must be a side $i < 2$ such that $\Fcal^{[i,P]}$ is 2-generic.
%
%
\begin{definition}
Let $c = (\sigma^0_s, \sigma^1_s, X_s, C, U_s : s < k) \in \Pb_1$,
$s < k$, and let $\Phi_{e_0}(G, n, m)$ and $\Phi_{e_1}(G, n, m)$ be two $\Delta_0$ formulas. Define the relation
$$
c \qvdash_s (\exists n)(\forall m)\Phi_{e_0}(G, n, m) \vee (\exists n)(\forall m)\Phi_{e_1}(G, n, m)
$$
to hold if for every $Z^0 \cup Z^1 = U_s$, there is finite set $F \subseteq C$, some $\rho_0, \dots, \rho_{a-1} \subseteq Z^0$ and $n_0, \dots, n_{a-1} \in \omega$, some $\mu_0, \dots, \mu_{b-1} \subseteq Z^1$ and $u_0, \dots, u_{b-1} \in \omega$ such that 
$$
\bigcap_{e \in F} \Ucal_e \bigcap_{j < a} \Ucal_{\zeta(e_0, \sigma^0_s \cup \rho_j, n_j)} \bigcap_{j < b} \Ucal_{\zeta(e_1, \sigma^1_s \cup \mu_j, u_j)}
$$
is not a largeness class.
\end{definition}

\begin{lemma}\label{lem:rt12-forcing-question-complexity}
Let $s$ be a branch in $c \in \Pb_1$ and let $\Phi_{e_0}(G, n, m)$ and $\Phi_{e_1}(G, n, m)$ be two $\Delta_0$ formulas. The relation 
$$
c \qvdash_s (\exists n)(\forall m)\Phi_{e_0}(G, n, m) \vee (\exists n)(\forall m)\Phi_{e_1}(G, n, m)
$$
is $\Sigma^0_1(\Mcal)$.
\end{lemma}
\begin{proof}
By compactness, the relation holds if there is a finite set $E \subseteq U_s$ such that for every $E_0 \cup E_1 = E$, there is some finite set $F \subseteq C$, some $\rho_0, \dots, \rho_{a-1} \subseteq E_0$ and $n_0, \dots, n_{a-1} \in \omega$, some $\mu_0, \dots, \mu_{b-1} \subseteq E_1$ and $u_0, \dots, u_{b-1} \in \omega$ such that 
$$
\bigcap_{e \in F} \Ucal_e \bigcap_{j < a} \Ucal_{\zeta(e_0, \sigma^0_s \cup \rho_j, n_j)} \bigcap_{j < b} \Ucal_{\zeta(e_1, \sigma^1_s \cup \mu_j, u_j)}
$$
is not a largeness class.
By Lemma~\ref{lem:largeness-class-complexity}, not being a largeness class is $\Sigma^0_2$. The overall relation is $\Sigma^0_1(C \oplus U_s \oplus \emptyset')$, hence is $\Sigma^0_1(\Mcal)$.
\end{proof}

\begin{lemma}\label{lem:rt12-forcing-question-spec}
Let $s$ be a branch in $c \in \Pb_1$ and let $\Phi_{e_0}(G, n, m)$ and $\Phi_{e_1}(G, n, m)$ be two $\Delta_0$ formulas.
\begin{itemize}
	\item[(a)] If $c \qvdash_s (\exists n)(\forall m)\Phi_{e_0}(G, n, m) \vee (\exists n)(\forall m)\Phi_{e_1}(G, n, m)$, then there is an $s$-extension $d \leq_f c$ such that for every valid branch $t$ in $d$ for which $f(t) = s$, there is some $i < 2$ such that $d^{[i,t]} \Vdash (\exists n)(\forall m)\Phi_{e_i}(G, n, m)$.
	\item[(b)] If $c \nqvdash_s (\exists n)(\forall m)\Phi_{e_0}(G, n, m) \vee (\exists n)(\forall m)\Phi_{e_1}(G, n, m)$, then there is an $s$-extension $d \leq_f c$ such that for every valid branch $t$ in $d$ for which $f(t) = s$, there is some $i < 2$ such that $d^{[i,t]} \Vdash (\forall n)(\exists m)\neg \Phi_{e_i}(G, n, m)$. 
\end{itemize}
Moreover, an index of $d$ can be found $A \oplus \Mcal$-uniformly in an index of $c$, $s$, $e_0$ and $e_1$ and the knowledge of which case holds.
\end{lemma}
\begin{proof}
Say $c = (\sigma^0_s, \sigma^1_s, X_s, C, U_s : s < k)$. 

(a) Let $Z^0 = U_s \cap A^0$ and $Z^1 = U_s \cap A^1$. Unfolding the definition of the forcing question, there is finite set $F \subseteq C$, some $\rho_0, \dots, \rho_{a-1} \subseteq Z^0$ and $n_0, \dots, n_{a-1} \in \omega$, some $\mu_0, \dots, \mu_{b-1} \subseteq Z^1$ and $u_0, \dots, u_{b-1} \in \omega$ such that 
$$
\bigcap_{e \in F} \Ucal_e \bigcap_{j < a} \Ucal_{\zeta(e_0, \sigma^0_s \cup \rho_j, n_j)} \bigcap_{j < b} \Ucal_{\zeta(e_1, \sigma^1_s \cup \mu_j, u_j)}
$$
is not a largeness class. Given $\ell \in \omega$, let $\Ccal_\ell$ be the $\Pi^{0,X_s}_1$ class of all $\ell$-covers of $\omega$ $R_0, \dots, R_{\ell-1}$ such that for every $t < \ell$, $R_t \not \in \bigcap_{e \in F} \Ucal_e \bigcap_{j < a} \Ucal_{\zeta(e_0, \sigma^0_s \cup \rho_j, n_j)} \bigcap_{j < b} \Ucal_{\zeta(e_1, \sigma^1_s \cup \mu_j, u_j)}$.
	By assumption, $\Ccal_\ell \neq \emptyset$ for some $\ell \in \omega$. By the low basis theorem, pick some $\ell$-cover of $\omega$ $R_0, \dots, R_{\ell-1}$ in $\Ccal_\ell$
	which is low over $X_s$.
	
	Define the $\Pb_1$-condition $d = (\tau^0_s, \tau^1_s, Y_s, C, V_s : s < k + \ell - 1)$ obtained from $c$ by splitting the branch $s$ into $\ell$ branches $s_0, \dots, s_{\ell-1}$, and leaving the other branches unchanged.
	For each $t < \ell$, let $V_{s_t} = U_s \cap R_t$, $Y_{s_t} = X_s \cap R_t$. If $R_t \not \in \Ucal_e$ for some $e \in C$, then $\tau^0_{s_t} = \sigma^0_s$ and $\tau^1_{s_t} = \sigma^1_s$. If $R_t \not \in \Ucal_{\zeta(e_0, \sigma^0_s \cup \rho_j, n_j)}$ for some $j < a$, let $\tau^1_{s_t} = \sigma^1_s$ and $\tau^0_{s_t} = \sigma^0_s \cup \rho_j$. If $R_t \not \in \Ucal_{\zeta(e_1, \sigma^1_s \cup \mu_j, u_j)}$ for some $j < b$, let $\tau^0_{s_t} = \sigma^0_s$ and $\tau^1_{s_t} = \sigma^1_s \cup \mu_j$.
	
	The $\Pb_1$-condition $d$ is by construction an $s$-extension of $c$. We now claim that for every valid branch $s_t$ in $d$, there is some $i < 2$ such that $d^{[i,s_t]} \Vdash (\exists n)(\forall m)\Phi_{e_i}(G, n, m)$. Since $s_t$ is valid in $d$, $R_t \cap U_s \in \bigcap_{e \in C} \Ucal_e$. In particular, $R_t \in \bigcap_{e \in C} \Ucal_e$, so there is some $i < 2$ such that $R_t \not \in \Ucal_{\zeta(e_i, \tau^i_{s_t}, n)}$  for some $n$.
		Unfolding the definition, $(\forall \tau \subseteq R_t)(\forall m)\Phi_{e_i}(\tau^i_{s_t} \cup \tau, n, m)$, so $d^{[i,s_t]} \Vdash (\exists n)(\forall m)\Phi_{e_i}(G, n, m)$. 

(b) Let $\Dcal$ be the class of all $Z^0 \oplus Z^1$ such that $Z^0 \cup Z^1 = U_s$ such that for every finite set $F \subseteq C$, every $\rho_0, \dots, \rho_{a-1} \subseteq Z^0$ and every $n_0, \dots, n_{a-1} \in \omega$, every $\mu_0, \dots, \mu_{b-1} \subseteq Z^1$ and every $u_0, \dots, u_{b-1} \in \omega$,  $\bigcap_{e \in F} \Ucal_e \bigcap_{j < a} \Ucal_{\zeta(e_0, \sigma^0_s \cup \rho_j, n_j)} \bigcap_{j < b} \Ucal_{\zeta(e_1, \sigma^1_s \cup \mu_j, u_j)}$ is a largeness class. By Lemma~\ref{lem:largeness-class-complexity}, being a largeness class for a $\Sigma^0_1$ class is $\Pi^0_2$, hence $\Pi^0_1(\emptyset')$. Since $U_s$ and $\emptyset'$ both belong to $\Mcal$, the class $\Dcal$ is $\Pi^0_1(\Mcal)$. Since $\Mcal \models \wkl$, there is some $Z^0 \oplus Z^1 \in \Dcal \cap \Mcal$. Let $D = C \cup \{ \zeta(e_i, \sigma^i_s \cup \rho, n) : i < 2, \rho \subseteq X_s \cap Z^i, n \in \omega)$. 
By Lemma~\ref{lem:decreasing-largeness-yields-largeness}, $\bigcap_{e \in D} \Ucal_e$ is a largeness class.

	
	Define the $\Pb_1$-condition $d = (\tau^0_s, \tau^1_s, Y_s, D, V_s : s < k +1)$ obtained from $c$ by splitting the branch $s$ into 2 branches $s_0, s_1$, and leaving the other branches unchanged. For each $i < 2$, let $Y_{s_i} = X_s$, $V_{s_i} = V \cap Z^i$, $\tau^0_{s_0} = \sigma^0_s$ and $\tau^1_{s_1} = \sigma^1_s$.
	For each $i < 2$, if $s_i$ is valid in $d$, then $d^{[i,s_i]} \Vdash (\forall n)(\exists m)\neg \Phi_{e_i}(G, n, m)$. This completes the proof of the lemma.
\end{proof}

\begin{lemma}
Let $\Fcal$ be a sufficiently generic $\Pb_1$-filter, and let $P$ be an $\Fcal$-projector.
There is some $i < 2$ such that $\Fcal^{[i,P]}$ is a 2-generic $\Qb_1$-filter.
\end{lemma}
\begin{proof}
By Lemma~\ref{lem:rt12-forcing-question-spec}, for every $\Pb_1$-condition $c \in \Fcal$ and every pair of $\Sigma^0_2$ formulas $\varphi_0(G), \varphi_1(G)$, there is a $P(c)$-extension $d$ of $c$
such that for every valid branch $t$ in $d$ refining the branch $P(c)$ of $c$,
$d^{[i,t]} \Vdash \varphi_i(G)$ or $d^{[i,t]} \Vdash \neg \varphi_i(G)$ for some $i < 2$.
Since $\Fcal$ is sufficiently generic, there is such an extension $d \in \Fcal$,
and since $P(d)$ is valid in $d$ and refines the branch $P(c)$ of $c$,
$d^{[i,P(d)]} \Vdash \varphi_i(G)$ or $d^{[i,P(d)]} \Vdash \neg \varphi_i(G)$ for some $i < 2$.
By a pairing argument, there is some $i < 2$ such that for every $\Sigma^0_2$
formula $\varphi(G)$, there is some $c \in \Fcal$ such that $c^{[i,P(c)]} \Vdash \varphi(G)$
or $c^{[i,P(c)]} \Vdash \neg \varphi(G)$. Therefore $\Fcal^{[i,P]}$ is 2-generic.
Since $P$ is an $\Fcal$-projector, $\Fcal^{[i,P]} \subseteq \Qb_1$.
Since $\Fcal$ is a filter, by Lemma~\ref{lem:pb1-to-qb1-compatibility},
so is $\Fcal^{[i,P]}$.
\end{proof}

\section{Applications}\label{sect:applications}

In this section, we apply the framework developed in section~\ref{sect:forcing-rt12} to derive our main theorems. 

\subsection{Preservation of non-$\Sigma^0_2$ definitions}

Our first application shows the existence, for every instance of the pigeonhole principle, of a solution which does not collapse the definition of a non-$\Sigma^0_2$ set into a $\Sigma^0_2$ one. This corresponds to preservation of one non-$\Sigma^0_2$ definition, following the terminology of Wang~\cite{Wang2014Definability}.

\begin{theorem}\label{thm:rt12-preservation-non-sigma2}
Fix a non-$\Sigma^0_2$ set $B$. For every set $A$, there is an infinite set $H \subseteq A$
or $H \subseteq \overline{A}$ such that $B$ is not $\Sigma^{0,H}_2$.
\end{theorem}

Fix $B$ and $A$, and let $A^0 = \overline{A}$ and $A^1 = A$.
By Wang~\cite[Theorem 3.6.]{Wang2014Definability}, there is a countable Turing ideal $\Mcal \models \wkl$ such that $B$ is not $\Sigma^0_1(\Mcal)$ and $\emptyset' \in \Mcal$. We build our infinite set by the notion of forcing $\Pb_1$ within the Turing ideal $\Mcal$.
Fix an enumeration $\varphi_0(G, n), \varphi_1(G, n)$ of all $\Sigma^0_2$ formulas with one set parameter $G$ and one integer parameter $n$.

\begin{lemma}\label{lem:preservation-non-sigma2-density}
Let $\Fcal$ be a sufficiently generic $\Pb_1$-filter and $P$ be an $\Fcal$-projector.
For every pair of $\Sigma^0_2$ formulas $\varphi_0(G, n)$ and $\varphi_1(G, n)$,
there is some $i < 2$ and some $p \in \Fcal^{[i,P]}$ such that
$$
(\exists n \not \in B) p \Vdash \varphi_i(G, n) \vee (\exists n \in B) p \Vdash \neg \varphi_i(G, n)
$$
\end{lemma}
\begin{proof}
Fix some $c \in \Fcal$, and let $s = P(c)$.
Let $W = \{ n : c \qvdash_s \varphi_0(G, n) \vee \varphi_1(G, n) \}$.
By Lemma~\ref{lem:rt12-forcing-question-complexity}, the set $W$ is $\Sigma^0_2$, therefore $W \neq B$. Let $n \in W \Delta B = (W - B) \cup (B - W)$. We have two cases.

Case 1: $n \in W - B$, then $c \qvdash_s \varphi_0(G, n) \vee \varphi_1(G, n)$. By Lemma~\ref{lem:rt12-forcing-question-spec}(a), there is an $s$-extension $d$ of $c$ such that for every valid branch $t$ of $d$ refining the branch $s$ of $c$, $d^{[i,t]} \Vdash \varphi_i(G,n)$ for some $i < 2$.

Case 2: $n \in B - W$, then $c \nqvdash_s \varphi_0(G, n) \vee \varphi_1(G, n)$.
By Lemma~\ref{lem:rt12-forcing-question-spec}(b), there is an $s$-extension $d$ of $c$ such that for every valid branch $t$ of $d$ refining the branch $s$ of $c$, $d^{[i,t]} \Vdash \neg \varphi_i(G,n)$ for some $i < 2$.

By genericity of $\Fcal$, there is such an extension $d \in \Fcal$. Let $P(d) = t$.
Since $t$ is valid in $d$, then either $n \not \in B$ and $d^{[i,t]} \Vdash \varphi_i(G,n)$,
or $n \in B$ and $d^{[i,t]} \Vdash \neg \varphi_i(G,n)$ for some $i < 2$.
\end{proof}

We are now ready to prove Theorem~\ref{thm:rt12-preservation-non-sigma2}.

\begin{proof}[Proof of Theorem~\ref{thm:rt12-preservation-non-sigma2}]
Let $\Fcal$ be a sufficiently generic $\Pb_1$-filter and $P$ be an $\Fcal$-projector.
By Lemma~\ref{lem:preservation-non-sigma2-density}, and by a pairing argument, there is some $i < 2$
such that for every $\Sigma^0_2$ formula $\varphi(G, n)$,
there is some $p \in \Fcal^{[i,P]}$ such that 
$$
(\exists n \not \in B) p \Vdash \varphi(G, n) \vee (\exists n \in B) p \Vdash \neg \varphi(G, n)
$$
In particular, $\Gcal = \Fcal^{[i,P]}$ is 2-generic, so by Lemma~\ref{lem:qb1-generic-yields-infinite}, $G_\Gcal$ is infinite, and by Lemma~\ref{lem:qb1-forcing-holds}, $B$ is not $\Sigma^0_2$. By definition of $\Pb_1$, $G_\Gcal \subseteq A^0$ or $G_\Gcal \subseteq A^1$. This completes the proof of Theorem~\ref{thm:rt12-preservation-non-sigma2}.
\end{proof}

The following corollary would correspond to strong jump cone avoidance of $\rt^1_2$,
following the terminology of Wang~\cite{Wang2014Some}.

\begin{corollary}\label{cor:rt12-jump-cone-avoidance}
Fix a non-$\Delta^0_2$ set $B$. For every set $A$, there is an infinite set $H \subseteq A$
or $H \subseteq \overline{A}$ such that $B$ is not $\Delta^{0,H}_2$.
\end{corollary}
\begin{proof}
Given a non-$\Delta^0_2$ set $B$, either $B$ or $\overline{B}$ is not $\Sigma^0_2(G_\Gcal)$.
By Theorem~\ref{thm:rt12-preservation-non-sigma2}, for every set $A$, 
there is an infinite set $H \subseteq A$
or $H \subseteq \overline{A}$ such that either $B$ or $\overline{B}$ is not $\Sigma^{0,H}_2$,
hence such that $B$ is not $\Delta^{0,H}_2$.
\end{proof}

The second author asked in~\cite[Question 2.7]{Patey2016Open} whether
there is a set such that every infinite subset of it or its complement is of high degree.
We answer negatively.

\begin{corollary}\label{cor:rt12-non-high}
For every set $A$, there is an infinite set $H \subseteq A$
or $H \subseteq \overline{A}$ of non-high degree.
\end{corollary}
\begin{proof}
Apply Corollary~\ref{cor:rt12-jump-cone-avoidance} to $B = \emptyset''$.
\end{proof}

\subsection{Preservation of $\Delta^0_2$ hyperimmunities}

Our second application concerns the ability to prevent solutions from computing fast-growing functions. By Martin's theorem~\cite{Martin1966Classes}, a set is of high degree iff it computes a function dominating every computable function. Therefore, Corollary~\ref{cor:rt12-non-high} already shows that an instance cannot force its solutions to compute arbitrarily fast growing functions. We shall now refine this result by proving that an instance cannot help dominating a fixed non-computably dominated function. Recall the definition of hyperimmunity.

\begin{definition}
A function $f$ \emph{dominates} a function $g$ if $f(x) \geq g(x)$
for every $x$. A function $f$ is \emph{$X$-hyperimmune} if it is not dominated by any $X$-computable function.
\end{definition}

\begin{theorem}\label{thm:rt12-preservation-delta2-hyperimmunity}
Fix a $\emptyset'$-hyperimmune function $f$. For every set $A$, there is an infinite set $H \subseteq A$
or $H \subseteq \overline{A}$ such that $f$ is $H'$-hyperimmune.
\end{theorem}

Fix $f$ and $A$, and let $A^0 = \overline{A}$ and $A^1 = A$.
By Jockusch and Soare~\cite{Jockusch197201}, there is a countable Turing ideal $\Mcal \models \wkl$ such that $f$ is $X$-hyperimmune for every $X \in \Mcal$, and $\emptyset' \in \Mcal$. We build our infinite set by the notion of forcing $\Pb_1$ within the Turing ideal $\Mcal$.

\begin{lemma}\label{lem:rt12-forcing-question-compact}
Let $s$ be a branch in $c \in \Pb_1$ and let $\Phi_{e_0}(G, n)$ and $\Phi_{e_1}(G, n)$ be two $\Sigma^0_2$ formulas. Suppose that 
$$
c \qvdash_s (\exists n)\Phi_{e_0}(G, n) \vee (\exists n)\Phi_{e_1}(G, n)
$$
then there is a finite set $U$ such that
$$
c \qvdash_s (\exists n \in U)\Phi_{e_0}(G, n) \vee (\exists n \in U)\Phi_{e_1}(G, n)
$$
\end{lemma}
\begin{proof}
Say $c = (\sigma^0_s, \sigma^1_s, X_s, C, U_s : s < k)$
and suppose  $c \qvdash_s (\exists n)\Phi_{e_0}(G, n) \vee (\exists n)\Phi_{e_1}(G, n)$ holds.
Then by compactness, there is a finite set $E \subseteq U_s$ such that for every $E_0 \cup E_1 = E$, there is some finite set $F \subseteq C$, some $\rho_0, \dots, \rho_{a-1} \subseteq E_0$ and $n_0, \dots, n_{a-1} \in \omega$, some $\mu_0, \dots, \mu_{b-1} \subseteq E_1$ and $u_0, \dots, u_{b-1} \in \omega$ such that 
$$
\bigcap_{e \in F} \Ucal_e \bigcap_{j < a} \Ucal_{\zeta(e_0, \sigma^0_s \cup \rho_j, n_j)} \bigcap_{j < b} \Ucal_{\zeta(e_1, \sigma^1_s \cup \mu_j, u_j)}
$$
is not a largeness class.
Then letting $U = \{n_0, \dots, n_{a-1}\} \cup \{u_0, \dots, u_{b-1}\}$,
$c \qvdash_s (\exists n \in U)\Phi_{e_0}(G, n) \vee (\exists n \in U)\Phi_{e_1}(G, n)$ also holds.
\end{proof}

\begin{lemma}\label{lem:preservation-delta2-hyperimmunity-density}
Let $\Fcal$ be a sufficiently generic $\Pb_1$-filter and $P$ be an $\Fcal$-projector.
For every pair of Turing functionals $\Phi_0$, $\Phi_1$, 
there is some $i < 2$, some $n$ and some $p \in \Fcal^{[i,P]}$ such that
$$
p \Vdash \Phi_i(G',n)\uparrow \vee p \Vdash \Phi_i(G',n)\downarrow < f(n)
$$
\end{lemma}
\begin{proof}
Fix some $c \in \Fcal$, and let $s = P(c)$.
Let $g$ be the partial $\Mcal$-computable function which on input $n$ searches for a finite set $U$ such that
$$
c \qvdash_s \Phi_0(G', n)\downarrow \in U \vee \Phi_1(G', n)\downarrow \in U
$$
If found, $g(n) = \max U$, otherwise $g(n)\uparrow$.
We have two cases.

Case 1: $g$ is total. Since $f$ is $\Mcal$-hyperimmune, there is some $n$ such that $g(n) < f(n)$. Let $U$ be the finite set witnessing $g(n)\downarrow$. Then 
$c \qvdash_s \Phi_0(G', n)\downarrow \in U \vee \Phi_1(G', n)\downarrow \in U$. By Lemma~\ref{lem:rt12-forcing-question-spec}(a), there is an $s$-extension $d$ of $c$ such that for every valid branch $t$ of $d$ refining the branch $s$ of $c$, $d^{[i,t]} \Vdash \Phi_i(G', n)\downarrow \in U$ for some $i < 2$.

Case 2: $g$ is partial, say $g(n)\uparrow$ for some $n$. Then by Lemma~\ref{lem:rt12-forcing-question-compact},
$c \nqvdash_s \Phi_0(G', n)\downarrow \vee \Phi_1(G', n)\downarrow$. By Lemma~\ref{lem:rt12-forcing-question-spec}(b), there is an $s$-extension $d$ of $c$ such that for every valid branch $t$ of $d$ refining the branch $s$ of $c$, $d^{[i,t]} \Vdash \Phi_i(G', n) \uparrow$ for some $i < 2$.

By genericity of $\Fcal$, there is such an extension $d \in \Fcal$. Let $P(d) = t$.
Since $t$ is valid in $d$, then either $d^{[i,t]} \Vdash \Phi_i(G', n)\uparrow$,
or $d^{[i,t]} \Vdash \Phi_i(G', n)\downarrow \in U$ for some $i < 2$ and some set $U < f(n)$.
\end{proof}

\begin{proof}[Proof of Theorem~\ref{thm:rt12-preservation-delta2-hyperimmunity}]
Let $\Fcal$ be a sufficiently generic $\Pb_1$-filter and $P$ be an $\Fcal$-projector.
By Lemma~\ref{lem:preservation-non-sigma2-density}, and by a pairing argument, there is some $i < 2$
such that for every Turing functional $\Phi$,
there is some $p \in \Fcal^{[i,P]}$ such that 
$$
p \Vdash \Phi(G',n)\uparrow \vee p \Vdash \Phi(G',n)\downarrow < f(n)
$$
In particular, $\Gcal = \Fcal^{[i,P]}$ is 2-generic, so by Lemma~\ref{lem:qb1-generic-yields-infinite}, $G_\Gcal$ is infinite, and by Lemma~\ref{lem:qb1-forcing-holds}, $f$ is $G_\Gcal'$-hyperimmune. By definition of $\Pb_1$, $G_\Gcal \subseteq A^0$ or $G_\Gcal \subseteq A^1$. This completes the proof of Theorem~\ref{thm:rt12-preservation-delta2-hyperimmunity}.
\end{proof}

Note that Theorem~\ref{thm:rt12-preservation-delta2-hyperimmunity}
cannot be extended to preservation of two $\emptyset'$-hyperimmune 
functions simultaneously, as there exists a bi-infinite set $A$
such that, $p_A$ and $p_{\overline{A}}$ are both $\emptyset'$-hyperimmune,
where $p_X$ is the function which on input $n$ returns the $n+1$st element of $X$.
For every infinite subset $H$ of $A$, $p_H$ dominates $p_A$, and for every infinite subset $H$
of $\overline{A}$ dominates $p_{\overline{A}}$. In both cases, $p_A$ and $p_{\overline{A}}$ cannot both be $H'$-hyperimmune for any solution $H$.

\subsection{Low${}_3$ solutions}

An effectivization of the forcing construction enables us to obtain lowness results for the infinite pigeonhole principle. The existence of low${}_2$ solutions for $\Delta^0_2$ sets, and of low${}_2$ cohesive sets for computable sequences of sets, was proven by Cholak, Jockusch and Slaman~\cite[sections 4.1 and 4.2]{Cholak2001strength}. The existence of low${}_3$ cohesive sets for $\Delta^0_2$ sequences of sets was proven by Wang~\cite[Theorem 3.4]{Wang2014Cohesive}. Wang~\cite[Questions 6.1 and 6.2]{Wang2014Cohesive} and the second author~\cite[Question 5.4]{Patey2016Open} asked whether such results can be generalized for every $\Delta^0_{n+1}$ instances of the pigeonhole and every $\Delta^0_n$ instances of cohesiveness. We answer positively to both questions in the case $n = 2$.

A set $Q$ is of \emph{PA degree} relative to $X$ (written $Q \gg X$) if it computes a member of every non-empty $\Pi^0_1$ class $\Ccal \subseteq 2^\omega$.

\begin{theorem}\label{thm:rt12-delta3-PA-double-jump-solution}
For every $\Delta^0_3$ set $A$ and every $Q \gg \emptyset''$, there is an infinite set $H \subseteq A$
or $H \subseteq \overline{A}$ such that $H'' \leq Q$.
\end{theorem}
\begin{proof}
By Simpson~\cite[Lemma VIII.2.9]{Simpson2009Subsystems} and the relativized low basis theorem~\cite{Jockusch197201}, there is a Turing ideal $\Mcal \models \wkl$ containing $\emptyset'$, and countable coded by a set $B$ such that $B' \leq_T \emptyset''$. Let $A^0 = \overline{A}$ and $A^1 = A$. Consider the notion of forcing $\Pb_1$ within the Turing ideal $\Mcal$.

Define an infinite decreasing sequence of $\Pb_1$-conditions $c_0 \geq_{f_0} c_1 \geq_{f_1} \dots$ such that for every $n$, letting $c_n = (\sigma^0_{s,n}, \sigma^1_{s,n}, X_{s,n}, C_n, U_{s,n} : s < k_n)$, and every $s < k_n$, either $s$ is not valid in $c_n$, or there is some $i < 2$ such that
$$
c^{[i,s]}_n \Vdash (\exists a)(\forall b)\Phi_n(G, a, b) \mbox{ or } c^{[i,s]}_n \Vdash (\forall a)(\exists b)\neg \Phi_n(G, a, b)
$$
We claim that there is a $\Pi^0_2(B)$, hence $\Pi^0_3$ such descending sequence.
Indeed, by Lemma~\ref{lem:largeness-class-complexity}, at a given stage $n$, we have a condition $c_n$,
it is $\Pi^0_2(B)$ to determine which branches are valid. Given a valid branch $s$, by Lemma~\ref{lem:rt12-forcing-question-spec}, we can find an $s$-extension $d$ such that 
for every valid branch $t$ of $d$ refining the branch $s$ of $c_n$, there is some side $i < 2$ such that the property holds. 
Such extension is obtained by checking whether the relation $\qvdash_s$ holds, which is $\Sigma^0_1(B)$ by Lemma~\ref{lem:rt12-forcing-question-complexity}, and then finding $d$ $A \oplus \Mcal$-uniformly, hence $\emptyset^{(2)}$-uniformly.

The valid branches of the conditions in the sequence $c_0 \geq_{f_0} c_1 \geq_{f_1} \dots$ form a $\Pi^0_3$ tree whose branches are $\emptyset''$-bounded. Let $\Fcal = \{c_0, c_1, \dots \}$. Since $Q$ is of PA degree above $\emptyset^{(2)}$, $Q$ computes an $\Fcal$-projector $P$, that is, a function such that $P(c_n) < k_n$ is a valid branch in $c_n$, and $f_n(P(c_n)) = P(c_{n-1})$. By the pairing argument, there is a side $i < 2$ such that $\Gcal = \Fcal^{[i,P]}$ is 2-generic. The set $G_\Gcal$ is $Q$-computable. By definition of $\Pb_1$, $G_\Gcal \subseteq A^0$ or $G_\Gcal \subseteq A^1$. By Lemma~\ref{lem:qb1-generic-yields-infinite}, $G_\Gcal$ is infinite, and by Lemma~\ref{lem:qb1-forcing-holds}, $(\exists a)(\forall b)\Phi_n(G_\Gcal, a, b)$ holds iff $c^{[i,P(c_n)]}_n \Vdash (\exists a)(\forall b)\Phi_n(G_\Gcal, a, b)$. Therefore, $G_\Gcal^{(2)} \leq_T Q$. This completes the proof of Theorem~\ref{thm:rt12-delta3-PA-double-jump-solution}.
\end{proof}

\begin{corollary}\label{cor:rt12-delta3-low3-solution}
For every $\Delta^0_3$ set $A$, there is an infinite set $H \subseteq A$
or $H \subseteq \overline{A}$ of low${}_3$ degree.
\end{corollary}
\begin{proof}
By the relativized low basis theorem~\cite{Jockusch197201}, there is some $Q \gg \emptyset''$
such that $Q' \leq_T \emptyset^{(3)}$. By Theorem~\ref{thm:rt12-delta3-PA-double-jump-solution},
there is an infinite set $H \subseteq A$ or $H \subseteq \overline{A}$ such that $H'' \leq Q$.
In particular, $H^{(3)} \leq_T Q' \leq_T \emptyset^{(3)}$.
\end{proof}

\section{Another proof of Liu's theorem}

The original computable analysis of Ramsey's theorem from Jockusch~\cite{Jockusch1972Ramseys} produces solutions of complete degree, while the more effective proofs of Ramsey's theorem for pairs of Seetapun~\cite{Seetapun1995strength} and Cholak, Jockusch and Slaman~\cite{Cholak2001strength} involve a compactness argument producing solutions of PA degree. Since complete degrees are also PA, it was a long-standing open question whether Ramsey's theorem for pairs implies weak K\"onig's lemma, until answered negatively by Liu~\cite{Liu2012RT22}. Later, Flood~\cite{Flood2012Reverse} identified the amount of compactness present in Ramsey's theorem for pairs as a principle called the Ramsey-type weak K\"onig's lemma. In this section, we reprove Liu's theorem by providing an alternative (and arguably simpler) proof of his main combinatorial theorem.

\begin{theorem}[Liu]\label{thm:rt12-strong-pa-avoidance}
For every set $A$, there is an infinite set $H \subseteq A$ or $H \subseteq \overline{A}$
of non-PA degree.
\end{theorem}

Before proving Theorem~\ref{thm:rt12-strong-pa-avoidance}, we deduce Liu's theorem from it.

\begin{theorem}[Liu]
For every computable coloring $f : [\omega]^2 \to 2$, there is an infinite 
$f$-homogeneous set of non-PA degree.
\end{theorem}
\begin{proof}
Let $\vec{R} = R_0, R_1, \dots$ be the computable sequence defined by $R_x = \{ y : f(x, y) = 1 \}$.
By Cholak, Jockusch and Slaman~\cite[Lemma 9.16]{Cholak2001strength}, there is an infinite $\vec{R}$-cohesive set $C = \{ x_0 < x_1 < \dots \}$ of non-PA degree. This fact can be proven using a simple computable Mathias forcing. In particular, $(\forall x \in \omega)\lim_{y \in C} f(x, y)$ exists.
Let $A = \{ n \in \omega : \lim_m f(x_n, x_m) = 1 \}$. Note that $\overline{A} = \{ n \in \omega : \lim_m f(x_n, x_m) = 0 \}$. By a relativization of Theorem~\ref{thm:rt12-strong-pa-avoidance}, there is an infinite set $H \subseteq A$ or $H \subseteq \overline{A}$ such that $H \oplus C$ is of non-PA degree. Suppose that $H \subseteq A$, $H \oplus C$-computably thin out the set $\{ x_n : n \in H \}$ using a greedy argument to obtain an infinite $f$-homogeneous set for color~$1$. The case $H \subseteq \overline{A}$ is similar, and yields and infinite set $f$-homogeneous for color~$0$. 
\end{proof}

We now prove Theorem~\ref{thm:rt12-strong-pa-avoidance} using a variant of the notion of forcing designed in this paper. Fix a set $A$, and let $A^0 = A$ and $A^1 = \overline{A}$.

\begin{definition}
Let $\Pb$ denote the set of conditions $(\sigma_s, X_s, C : s < k)$ such that
\begin{itemize}
	\item[(a)] for every $s < k$, there is some $i < 2$ such that $\sigma_s \cup X_s \subseteq A^i$
	\item[(b)] $X_0, \dots, X_{k-1}$ is a $k$-cover of $\omega - \{0, \dots, \max_s |\sigma_s|\}$
	\item[(c)] $\bigcap_{e \in C} \Ucal_e$ is a largeness class containing only infinite sets
	\item[(d)] $C$ is computable
\end{itemize}
\end{definition}

Note that no effectiveness restriction is imposed on the reservoirs $\vec{X}$.
Their restriction is of combinatorial nature, since their union is required to be cofinite.

\begin{definition}
The partial order on $\Pb$ is defined by $(\tau_s, Y_s, D : s < \ell) \leq (\sigma_s, X_s, C : s < k)$ if there is a function $f : \ell \to k$ such that for every $s < \ell$, $\sigma_{f(s)} \preceq \tau_s$, $Y_s \subseteq X_{f(s)}$, $C \subseteq D$ and $\tau_s - \sigma_{f(s)} \subseteq X_{f(s)}$.
\end{definition}

Again, we write $d \leq_f c$ if $d \leq c$ is witnessed by the function $f$,
and say that the branch $s$ is valid in $c = (\sigma_s, X_s, C : s < k)$
if $X_s \in \bigcap_{e \in C} \Ucal_e$. We write $c^{[s]}$ for $(\sigma_s, X_s, C, X_s)$.
In other words, the branch $s$ is valid in $c$ if $c^{[s]} \in \Qb_1$.
The notion of projector is defined accordingly. 

\begin{lemma}\label{lem:rt12-non-pa-progress-infinity}
Let $c \in \Pb$ be a condition, and $s$ be such that $c^{[s]} \in \Qb_1$. Then for every $n$, there is a simple extension $d \leq c$ such that $d^{[s]} \Vdash (\exists m > n) m \in G$.
\end{lemma}
\begin{proof}
Say $c = (\sigma_s, X_s, C : s < k)$. 
Since $c^{[s]} \in \Qb_1$, then $X_s \in \bigcap_{e \in C} \Ucal_e$.
In particular, $X_s$ is infinite. Let $m \in X_s$ be such that $m > n$.
Then condition $d = (\tau_s, X_s \cap (m, \infty), C : s < k)$ defined by
$\tau_s = \sigma_s \cup \{m\}$, and $\tau_t = \sigma_t$ otherwise, is the desired simple extension of~$c$.
\end{proof}

\begin{definition}
Let $\xi : \omega \times 2^{<\omega} \to \omega$ be the computable function that takes as an index of a Turing functional $\Phi_e(G, n)$, a string $\sigma$ and which gives a code for the open set
$$
\{ X : (\exists \rho \subseteq X - \{ 0, \dots, |\sigma|\})(\exists n)\Phi_e(\sigma \cup \rho, n) \downarrow = \Phi_n(n) \}
$$
\end{definition}

\begin{lemma}\label{lem:rt12-non-pa-progress}
Let $s$ be a branch of a condition $c \in \Pb$,
and let $\Phi_e(G,n)$ be a Turing functional.
Then there is an $s$-extension $d \leq_f c$ such that for every valid branch $t$ in $d$
for which $f(t) = s$,
$$
d^{[t]} \Vdash (\exists n)\Phi_e(G, n) \downarrow = \Phi_n(n)
\mbox{ or } d^{[t]} \Vdash (\exists n)\Phi_e(G, n) \uparrow
$$
\end{lemma}
\begin{proof}
Say $c = (\sigma_s, X_s, C : s < k)$. Define the predicate $P(n, k, v)$ to hold if
$$
(\forall Z_0 \cup \dots \cup Z_{k-1} = \omega)(\exists j < k)Z_j \in \bigcap_{e \in C, e < k} \Ucal_e \mbox{ and } (\exists \rho \subseteq Z_j - \{ 0, \dots, |\sigma_s|\})\Phi_e(\sigma_s \cup \rho, n)\downarrow = v
$$

Suppose first that the following is true for every $k$:
$$
(\forall n)(\exists v < 2)P(n, k, v)
$$
Note that for every $k$, the set $W_k = \{ (n, v) : P(n, k, v) \}$ is c.e. Therefore, there must be some $n \in \omega$ such that $(n, \Phi_n(n)) \in W_k$, otherwise we would compute a $\{0,1\}$-valued DNC function.
Then in particular, 
$$
(\forall Z_0 \cup \dots \cup Z_{k-1} = \omega)(\exists j < k) Z_j \in \bigcap_{e \in C} \Ucal_e \cap \Ucal_{\xi(e, \sigma_s)}
$$
It follows that $\bigcap_{e \in C} \Ucal_e \cap \Ucal_{\xi(e, \sigma_s)}$ is a largeness class. If $X_s \not \in \bigcap_{e \in C} \Ucal_e \cap \Ucal_{\xi(e, \sigma_s)}$,
then $d = (\sigma_s, X_s, C \cup \{\xi(e, \sigma_s)\}, s < k)$ is an $s$-extension of $c$ on which the side $s$ is not valid in $d$, and we are done.
If $X_s \in \bigcap_{e \in C} \Ucal_e \cap \Ucal_{\xi(e, \sigma_s)}$, then there is some $\rho \subseteq X_s - \{ 0, \dots, |\sigma|\}$ such that $\Phi_e(\sigma \cup \rho, n) \downarrow = \Phi_n(n)$.
The condition $d = (\tau_s, X_s \setminus \{0, \dots, |\rho|\}, C : s < k)$ defined by $\tau_s = \sigma_s \cup \rho$ and leaving the other branches unchanged, is an $s$-extension of $c$ 
such that $d^{[s]} \Vdash (\exists n)\Phi_e(G, n) \downarrow = \Phi_n(n)$.

Suppose now that there is a $k$ such that
$$
(\exists n)(\forall v < 2)\neg P(n, k, v)
$$
In particular, for some $k$ and some $n$, we have $k$-covers $Z^0_0 \cup \dots \cup Z^0_{k-1} = \omega$ and $Z^1_0 \cup \dots \cup Z^1_{k-1} = \omega$ such that for $v < 2$, we have
$$
(\forall j < k)Z^v_j \not \in \bigcap_{e \in C, e < k} \Ucal_e \mbox{ or } 
(\forall \rho \subseteq Z^v_j - \{ 0, \dots, |\sigma_s|\})\Phi_e(\sigma_s \cup \rho, n)\uparrow \vee \Phi_e(\sigma_s \cup \rho, n)\downarrow = 1-v
$$
Let $d = (\tau_s, Y_s, C : s < \ell)$ be the $s$-extension of $c$ obtained from $c$ by forking the branch $s$ into $k^2$ parts $s_{0,0}, \dots, s_{k-1,k-1}$, such that
$\tau_{s_{j,t}} = \sigma_s$ and $Y_{s_{j,t}} = X_s \cap Z^0_j \cap Z^1_t$.
Note that $d^{[s_{j,t}]} \Vdash (\exists n)\Phi_e(G, n) \uparrow$ for every $j, t < k$.
This completes the proof of the lemma.
\end{proof}

\begin{proof}[Proof of Theorem~\ref{thm:rt12-strong-pa-avoidance}]
Let $\Fcal$ be a sufficiently generic $\Pb$-filter,
and let $P$ be an $\Fcal$-projector.
Let $\Gcal = \{ c^{[P(c)]} : c \in \Fcal \}$ and $G_\Gcal = \bigcup \{ \sigma : (\sigma, X, C, \omega) \in \Gcal \}$.
By construction, $G_\Gcal \subseteq A^i$ for some $i < 2$.
By Lemma~\ref{lem:rt12-non-pa-progress-infinity}, $G_\Gcal$ is infinite.
By Lemma~\ref{lem:rt12-non-pa-progress}, $G_\Gcal$ is of non-PA degree.
This completes the proof.
\end{proof}


\vspace{0.5cm}

\bibliographystyle{plain}
\bibliography{bibliography}

\end{document}